\documentclass{article}

\usepackage[utf8]{inputenc}

\usepackage{amsmath,amssymb,amsfonts,amscd, graphicx, latexsym, verbatim, multirow, color}
\usepackage{amsthm}
\usepackage{tikz}
\usetikzlibrary{shapes.geometric}
\usepackage{forest}
\usepackage{float}
\usepackage{graphicx}

\usepackage{mwe}

\usepackage{amsmath,amssymb,amsfonts,amscd, graphicx, latexsym, verbatim, multirow, color, float, enumitem}
\usepackage{caption, subcaption}
\usepackage{mathtools} 
\usepackage{pgf, tikz}
\usetikzlibrary{patterns}
\usetikzlibrary{decorations.shapes}
\newtheorem{theorem}{Theorem}[section]
\newtheorem{proposition}[theorem]{Proposition}

\theoremstyle{definition}
\newtheorem{definition}[theorem]{Definition}
\newtheorem{example}[theorem]{Example}

\newtheorem{remark}[theorem]{Remark}

\makeatletter
\@addtoreset{equation}{section}
\makeatother

\newcommand{\NN}{{\mathbb N}}

\newcommand{\ZZ}{{\mathbb Z}}

\newcommand{\RR}{{\mathbb R}}

\newcommand{\cF}{{\mathcal F}}

\newcommand{\cP}{{\mathcal P}}

\newcommand{\FAI}[4]{
\begin{scope}[xshift=#1cm, yshift=#2cm, rotate=#3, scale=#4]
\filldraw[color=yellow!50] (0,0) rectangle ({(1+sqrt(5))/2},{(1+sqrt(5))/2});
\draw[line width = 1pt] (0,0) -> ({(1+sqrt(5))/4},{(1+sqrt(5))/4}) -- ({(1+sqrt(5))/2},0);
\draw[line width = 1pt, ->] (0,0) -- ({(1+sqrt(5))/4},{(1+sqrt(5))/4});
\end{scope}}

\newcommand{\FAII}[4]{
\begin{scope}[xshift=#1cm, yshift=#2cm, rotate=#3, scale=#4]
\filldraw[color=yellow!50] (0,0) rectangle ({(1+sqrt(5))/2},{(1+sqrt(5))/2});
\draw[line width = 1pt] ({(1+sqrt(5))/2},0) -> ({(1+sqrt(5))/4},{(1+sqrt(5))/4}) -- ({(1+sqrt(5))/2},{(1+sqrt(5))/2});
\draw[line width = 1pt, ->] ({(1+sqrt(5))/2},0) -- ({(1+sqrt(5))/4},{(1+sqrt(5))/4});
\end{scope}}

\newcommand{\FAIII}[4]{
\begin{scope}[xshift=#1cm, yshift=#2cm, rotate=#3, scale=#4]
\filldraw[color=yellow!50] (0,0) rectangle ({(1+sqrt(5))/2},{(1+sqrt(5))/2});
\draw[line width = 1pt] ({(1+sqrt(5))/2},{(1+sqrt(5))/2}) -> ({(1+sqrt(5))/4},{(1+sqrt(5))/4}) -- (0,{(1+sqrt(5))/2});
\draw[line width = 1pt, ->] ({(1+sqrt(5))/2},{(1+sqrt(5))/2}) -- ({(1+sqrt(5))/4},{(1+sqrt(5))/4});

\end{scope}}

\newcommand{\FAIV}[4]{
\begin{scope}[xshift=#1cm, yshift=#2cm, rotate=#3, scale=#4]
\filldraw[color=yellow!50] (0,0) rectangle ({(1+sqrt(5))/2},{(1+sqrt(5))/2});;
\draw[line width = 1pt] (0,0) -- ({(1+sqrt(5))/4},{(1+sqrt(5))/4}) -- (0,{(1+sqrt(5))/2});
\draw[line width = 1pt, ->] (0,{(1+sqrt(5))/2}) -- ({(1+sqrt(5))/4},{(1+sqrt(5))/4});
\end{scope}}

\newcommand{\FARI}[4]{
\begin{scope}[xshift=#1cm, yshift=#2cm, rotate=#3, scale=#4]
\filldraw[color=gray!50] (0,0) rectangle ({(1+sqrt(5))/2},{(1+sqrt(5))/2});
\draw[line width = 1pt] (0,0) -> ({(1+sqrt(5))/4},{(1+sqrt(5))/4}) -- ({(1+sqrt(5))/2},0);
\draw[line width = 1pt, ->] ({(1+sqrt(5))/2},0) -- ({(1+sqrt(5))/4},{(1+sqrt(5))/4});
\end{scope}}

\newcommand{\FARII}[4]{
\begin{scope}[xshift=#1cm, yshift=#2cm, rotate=#3, scale=#4]
\filldraw[color=gray!50] (0,0) rectangle ({(1+sqrt(5))/2},{(1+sqrt(5))/2});
\draw[line width = 1pt] ({(1+sqrt(5))/2},0) -> ({(1+sqrt(5))/4},{(1+sqrt(5))/4}) -- ({(1+sqrt(5))/2},{(1+sqrt(5))/2});
\draw[line width = 1pt, ->] ({(1+sqrt(5))/2},{(1+sqrt(5))/2}) -- ({(1+sqrt(5))/4},{(1+sqrt(5))/4});
\end{scope}}

\newcommand{\FARIII}[4]{
\begin{scope}[xshift=#1cm, yshift=#2cm, rotate=#3, scale=#4]
\filldraw[color=gray!50] (0,0) rectangle ({(1+sqrt(5))/2},{(1+sqrt(5))/2});
\draw[line width = 1pt] ({(1+sqrt(5))/2},{(1+sqrt(5))/2}) -> ({(1+sqrt(5))/4},{(1+sqrt(5))/4}) -- (0,{(1+sqrt(5))/2});
\draw[line width = 1pt, ->] (0,{(1+sqrt(5))/2}) -- ({(1+sqrt(5))/4},{(1+sqrt(5))/4});
\end{scope}}

\newcommand{\FARIV}[4]{
\begin{scope}[xshift=#1cm, yshift=#2cm, rotate=#3, scale=#4]
\filldraw[color=gray!50] (0,0) rectangle ({(1+sqrt(5))/2},{(1+sqrt(5))/2});
\draw[line width = 1pt] (0,0) -- ({(1+sqrt(5))/4},{(1+sqrt(5))/4}) -- (0,{(1+sqrt(5))/2});
\draw[line width = 1pt, ->] (0,0) -- ({(1+sqrt(5))/4},{(1+sqrt(5))/4});
\end{scope}}

\newcommand{\FBI}[4]{
\begin{scope}[xshift=#1cm, yshift=#2cm, rotate=#3, scale=#4]
\filldraw[color=teal!50] (0,0) rectangle ({(1+sqrt(5))/2},1);
\draw[line width = 1pt] (0,0) -- ({(1+sqrt(5))/2},1);
\draw[line width = 1pt, ->] (0,0) -- ({(1+sqrt(5))/4},0.5);
\end{scope}}

\newcommand{\FBRI}[4]{
\begin{scope}[xshift=#1cm, yshift=#2cm, rotate=#3, scale=#4]
\filldraw[color=magenta!50] (0,0) rectangle ({(1+sqrt(5))/2},1);
\draw[line width = 1pt] ({(1+sqrt(5))/2},1) -- (0,0);
\draw[line width = 1pt, ->] ({(1+sqrt(5))/2},1) -- ({(1+sqrt(5))/4},0.5);
\end{scope}}

\newcommand{\FBII}[4]{
\begin{scope}[xshift=#1cm, yshift=#2cm, rotate=#3, scale=#4]
\filldraw[color=teal!50] (0,0) rectangle ({(1+sqrt(5))/2},1);
\draw[line width = 1pt] ({(1+sqrt(5))/2},0) -- (0,1);
\draw[line width = 1pt, ->] ({(1+sqrt(5))/2},0) -- ({(1+sqrt(5))/4},0.5);
\end{scope}}

\newcommand{\FBRII}[4]{
\begin{scope}[xshift=#1cm, yshift=#2cm, rotate=#3, scale=#4]
\filldraw[color=magenta!50] (0,0) rectangle ({(1+sqrt(5))/2},1);
\draw[line width = 1pt] ({(1+sqrt(5))/2},0) -- (0,1);
\draw[line width = 1pt, ->] (0,1) -- ({(1+sqrt(5))/4},0.5);
\end{scope}}

\newcommand{\FCI}[4]{
\begin{scope}[xshift=#1cm, yshift=#2cm, rotate=#3, scale=#4]
\filldraw[color=orange!50] (0,0) rectangle (1,{(1+sqrt(5))/2});
\draw[line width = 1pt] (0,{(1+sqrt(5))/2}) -- (1,0);
\draw[line width = 1pt, ->] (1,0) -- (0.5,{(1+sqrt(5))/4});

\end{scope}}

\newcommand{\FCRI}[4]{
\begin{scope}[xshift=#1cm, yshift=#2cm, rotate=#3, scale=#4]
\filldraw[color=olive!50] (0,0) rectangle (1,{(1+sqrt(5))/2});
\draw[line width = 1pt] (0,{(1+sqrt(5))/2}) -- (1,0);
\draw[line width = 1pt, ->] (0,{(1+sqrt(5))/2}) -- (0.5,{(1+sqrt(5))/4});
\end{scope}}

\newcommand{\FCII}[4]{
\begin{scope}[xshift=#1cm, yshift=#2cm, rotate=#3, scale=#4]
\filldraw[color=orange!50] (0,0) rectangle (1,{(1+sqrt(5))/2});
\draw[line width = 1pt] (0,0) -- (1,{(1+sqrt(5))/2});
\draw[line width = 1pt, ->] (0,0) -- (0.5,{(1+sqrt(5))/4});
\end{scope}}

\newcommand{\FCRII}[4]{
\begin{scope}[xshift=#1cm, yshift=#2cm, rotate=#3, scale=#4]
\filldraw[color=olive!50] (0,0) rectangle (1,{(1+sqrt(5))/2});
\draw[line width = 1pt] (0,0) -- (1,{(1+sqrt(5))/2});
\draw[line width = 1pt, ->] (1,{(1+sqrt(5))/2}) -- (0.5,{(1+sqrt(5))/4});
\end{scope}}

\newcommand{\FDI}[4]{
\begin{scope}[xshift=#1cm, yshift=#2cm, rotate=#3, scale=#4]
\filldraw[color=green!50] (0,0) rectangle (1,1);
\draw[line width = 1pt] (0,0) -> (0.5,0.5) -- (1,0);
\draw[line width = 1pt, ->] (0,0) -- (0.5,0.5);
\end{scope}}

\newcommand{\FDII}[4]{
\begin{scope}[xshift=#1cm, yshift=#2cm, rotate=#3, scale=#4]
\filldraw[color=green!50] (0,0) rectangle (1,1);

\draw[line width = 1pt] (1,0) -> (0.5,0.5) -- (1,1);
\draw[line width = 1pt, ->] (1,0) -- (0.5,0.5);
\end{scope}}

\newcommand{\FDIII}[4]{
\begin{scope}[xshift=#1cm, yshift=#2cm, rotate=#3, scale=#4]
\filldraw[color=green!50] (0,0) rectangle (1,1);
\draw[line width = 1pt] (1,1) -> (0.5,0.5) -- (0,1);
\draw[line width = 1pt, ->] (1,1) -- (0.5,0.5);

\end{scope}}

\newcommand{\FDIV}[4]{
\begin{scope}[xshift=#1cm, yshift=#2cm, rotate=#3, scale=#4]
\filldraw[color=green!50] (0,0) rectangle (1,1);
\draw[line width = 1pt] (0,0) -- (0.5,0.5) -- (0,1);
\draw[line width = 1pt, ->] (0,1) -- (0.5,0.5);
\end{scope}}

\newcommand{\FDRI}[4]{
\begin{scope}[xshift=#1cm, yshift=#2cm, rotate=#3, scale=#4]
\filldraw[color=blue!50] (0,0) rectangle (1,1);
\draw[line width = 1pt] (0,0) -> (0.5,0.5) -- (1,0);
\draw[line width = 1pt, ->] (1,0) -- (0.5,0.5);
\end{scope}}

\newcommand{\FDRII}[4]{
\begin{scope}[xshift=#1cm, yshift=#2cm, rotate=#3, scale=#4]
\filldraw[color=blue!50] (0,0) rectangle (1,1);
\draw[line width = 1pt] (1,0) -> (0.5,0.5) -- (1,1);
\draw[line width = 1pt, ->] (1,1) -- (0.5,0.5);
\end{scope}}

\newcommand{\FDRIII}[4]{
\begin{scope}[xshift=#1cm, yshift=#2cm, rotate=#3, scale=#4]
\filldraw[color=blue!50] (0,0) rectangle (1,1);
\draw[line width = 1pt] (1,1) -> (0.5,0.5) -- (0,1);
\draw[line width = 1pt, ->] (0,1) -- (0.5,0.5);
\end{scope}}

\newcommand{\FDRIV}[4]{
\begin{scope}[xshift=#1cm, yshift=#2cm, rotate=#3, scale=#4]
\filldraw[color=blue!50] (0,0) rectangle (1,1);
\draw[line width = 1pt] (0,0) -- (0.5,0.5) -- (0,1);
\draw[line width = 1pt, ->] (0,0) -- (0.5,0.5);
\end{scope}}

\newcommand{\SFAI}[4]{
\begin{scope}[xshift=#1cm, yshift=#2cm, rotate=#3, scale=#4]
\FARIV{0}{0}{0}{1}
\FBI{0}{1.61803399}{0}{1}
\FCRI{1.61803399}{0}{0}{1}
\FDIV{1.61803399}{1.61803399}{0}{1}
\end{scope}}

\newcommand{\SFAII}[4]{
\begin{scope}[xshift=#1cm, yshift=#2cm, rotate=#3, scale=#4]
\FAIII{0}{0}{0}{1}
\FBI{0}{1.61803399}{0}{1}
\FCI{1.61803399}{0}{0}{1}
\FDRIII{1.61803399}{1.61803399}{0}{1}
\end{scope}}

\newcommand{\SFAIII}[4]{
\begin{scope}[xshift=#1cm, yshift=#2cm, rotate=#3, scale=#4]
\FAII{0}{0}{0}{1}
\FBII{0}{1.61803399}{0}{1}
\FCRII{1.61803399}{0}{0}{1}
\FDRII{1.61803399}{1.61803399}{0}{1}
\end{scope}}

\newcommand{\SFAIV}[4]{
\begin{scope}[xshift=#1cm, yshift=#2cm, rotate=#3, scale=#4]
\FARI{0}{0}{0}{1}
\FBRII{0}{1.61803399}{0}{1}
\FCRII{1.61803399}{0}{0}{1}
\FDI{1.61803399}{1.61803399}{0}{1}
\end{scope}}

\newcommand{\SFARI}[4]{
\begin{scope}[xshift=#1cm, yshift=#2cm, rotate=#3, scale=#4]
\FAIV{0}{0}{0}{1}
\FBRI{0}{1.61803399}{0}{1}
\FCI{1.61803399}{0}{0}{1}
\FDRIV{1.61803399}{1.61803399}{0}{1}
\end{scope}}

\newcommand{\SFARII}[4]{
\begin{scope}[xshift=#1cm, yshift=#2cm, rotate=#3, scale=#4]
\FARIII{0}{0}{0}{1}
\FBRI{0}{1.61803399}{0}{1}
\FCRI{1.61803399}{0}{0}{1}
\FDIII{1.61803399}{1.61803399}{0}{1}
\end{scope}}

\newcommand{\SFARIII}[4]{
\begin{scope}[xshift=#1cm, yshift=#2cm, rotate=#3, scale=#4]
\FARII{0}{0}{0}{1}
\FBRII{0}{1.61803399}{0}{1}
\FCII{1.61803399}{0}{0}{1}
\FDII{1.61803399}{1.61803399}{0}{1}
\end{scope}}

\newcommand{\SFARIV}[4]{
\begin{scope}[xshift=#1cm, yshift=#2cm, rotate=#3, scale=#4]
\FAI{0}{0}{0}{1}
\FBII{0}{1.61803399}{0}{1}
\FCII{1.61803399}{0}{0}{1}
\FDRI{1.61803399}{1.61803399}{0}{1}
\end{scope}}

\newcommand{\SFBI}[4]{
\begin{scope}[xshift=#1cm, yshift=#2cm, rotate=#3, scale=#4]
\FAI{0}{0}{0}{1}
\FCII{1.61803399}{0}{0}{1}
\end{scope}}

\newcommand{\SFBRI}[4]{
\begin{scope}[xshift=#1cm, yshift=#2cm, rotate=#3, scale=#4]
\FARI{0}{0}{0}{1}
\FCRII{1.61803399}{0}{0}{1}
\end{scope}}

\newcommand{\SFBII}[4]{
\begin{scope}[xshift=#1cm, yshift=#2cm, rotate=#3, scale=#4]
\FAIII{0}{0}{0}{1}
\FCI{1.61803399}{0}{0}{1}
\end{scope}}

\newcommand{\SFBRII}[4]{
\begin{scope}[xshift=#1cm, yshift=#2cm, rotate=#3, scale=#4]
\FARIII{0}{0}{0}{1}
\FCRI{1.61803399}{0}{0}{1}
\end{scope}}

\newcommand{\SFCI}[4]{
\begin{scope}[xshift=#1cm, yshift=#2cm, rotate=#3, scale=#4]
\FAII{0}{0}{0}{1}
\FBII{0}{1.61803399}{0}{1}
\end{scope}}

\newcommand{\SFCRI}[4]{
\begin{scope}[xshift=#1cm, yshift=#2cm, rotate=#3, scale=#4]
\FARII{0}{0}{0}{1}
\FBRII{0}{1.61803399}{0}{1}
\end{scope}}

\newcommand{\SFCII}[4]{
\begin{scope}[xshift=#1cm, yshift=#2cm, rotate=#3, scale=#4]
\FARIV{0}{0}{0}{1}
\FBI{0}{1.61803399}{0}{1}
\end{scope}}

\newcommand{\SFCRII}[4]{
\begin{scope}[xshift=#1cm, yshift=#2cm, rotate=#3, scale=#4]
\FAIV{0}{0}{0}{1}
\FBRI{0}{1.61803399}{0}{1}
\end{scope}}

\newcommand{\SFDI}[4]{
\begin{scope}[xshift=#1cm, yshift=#2cm, rotate=#3, scale=#4]
\FAI{0}{0}{0}{1}
\end{scope}}

\newcommand{\SFDII}[4]{
\begin{scope}[xshift=#1cm, yshift=#2cm, rotate=#3, scale=#4]
\FAII{0}{0}{0}{1}
\end{scope}}

\newcommand{\SFDIII}[4]{
\begin{scope}[xshift=#1cm, yshift=#2cm, rotate=#3, scale=#4]
\FAIII{0}{0}{0}{1}
\end{scope}}

\newcommand{\SFDIV}[4]{
\begin{scope}[xshift=#1cm, yshift=#2cm, rotate=#3, scale=#4]
\FAIV{0}{0}{0}{1}
\end{scope}}

\newcommand{\SFDRI}[4]{
\begin{scope}[xshift=#1cm, yshift=#2cm, rotate=#3, scale=#4]
\FARI{0}{0}{0}{1}
\end{scope}}

\newcommand{\SFDRII}[4]{
\begin{scope}[xshift=#1cm, yshift=#2cm, rotate=#3, scale=#4]
\FARII{0}{0}{0}{1}
\end{scope}}

\newcommand{\SFDRIII}[4]{
\begin{scope}[xshift=#1cm, yshift=#2cm, rotate=#3, scale=#4]
\FARIII{0}{0}{0}{1}
\end{scope}}

\newcommand{\SFDRIV}[4]{
\begin{scope}[xshift=#1cm, yshift=#2cm, rotate=#3, scale=#4]
\FARIV{0}{0}{0}{1}
\end{scope}}

\newcommand{\SSFAI}[4]{
\begin{scope}[xshift=#1cm, yshift=#2cm, rotate=#3, scale=#4]
\SFARIV{0}{0}{0}{1}
\SFBI{0}{2.61803399}{0}{1}
\SFCRI{2.61803399}{0}{0}{1}
\SFDIV{2.61803399}{2.61803399}{0}{1}
\end{scope}}

\newcommand{\SSFAIV}[4]{
\begin{scope}[xshift=#1cm, yshift=#2cm, rotate=#3, scale=#4]
\SFARI{0}{0}{0}{1}
\SFBRII{0}{2.61803399}{0}{1}
\SFCRII{2.61803399}{0}{0}{1}
\SFDI{2.61803399}{2.61803399}{0}{1}
\end{scope}}

\newcommand{\SSFARII}[4]{
\begin{scope}[xshift=#1cm, yshift=#2cm, rotate=#3, scale=#4]
\SFARIII{0}{0}{0}{1}
\SFBRI{0}{2.61803399}{0}{1}
\SFCRI{2.61803399}{0}{0}{1}
\SFDIII{2.61803399}{2.61803399}{0}{1}
\end{scope}}

\newcommand{\SSFBII}[4]{
\begin{scope}[xshift=#1cm, yshift=#2cm, rotate=#3, scale=#4]
\SFAIII{0}{0}{0}{1}
\SFCI{2.61803399}{0}{0}{1}
\end{scope}}

\newcommand{\SSFBRII}[4]{
\begin{scope}[xshift=#1cm, yshift=#2cm, rotate=#3, scale=#4]
\SFARIII{0}{0}{0}{1}
\SFCRI{2.61803399}{0}{0}{1}
\end{scope}}

\newcommand{\SSFCII}[4]{
\begin{scope}[xshift=#1cm, yshift=#2cm, rotate=#3, scale=#4]
\SFARIV{0}{0}{0}{1}
\SFBI{0}{2.61803399}{0}{1}
\end{scope}}

\newcommand{\SSFDRI}[4]{
\begin{scope}[xshift=#1cm, yshift=#2cm, rotate=#3, scale=#4]
\SFARI{0}{0}{0}{1}
\end{scope}}

\newcommand{\SSSFARIV}[4]{
\begin{scope}[xshift=#1cm, yshift=#2cm, rotate=#3, scale=#4]
\SSFAI{0}{0}{0}{1}
\SSFBII{0}{4.23606798}{0}{1}
\SSFCII{4.23606798}{0}{0}{1}
\SSFDRI{4.23606798}{4.23606798}{0}{1}
\end{scope}}

\newcommand{\SSSFBI}[4]{
\begin{scope}[xshift=#1cm, yshift=#2cm, rotate=#3, scale=#4]
\SSFAI{0}{0}{0}{1}
\SSFCII{4.23606798}{0}{0}{1}
\end{scope}}

\newcommand{\SSSFCRI}[4]{
\begin{scope}[xshift=#1cm, yshift=#2cm, rotate=#3, scale=#4]
\SSFARII{0}{0}{0}{1}
\SSFBRII{0}{4.23606798}{0}{1}
\end{scope}}

\newcommand{\SSSFDIV}[4]{
\begin{scope}[xshift=#1cm, yshift=#2cm, rotate=#3, scale=#4]
\SSFAIV{0}{0}{0}{1}
\end{scope}}

\newcommand{\SSSSFAI}[4]{
\begin{scope}[xshift=#1cm, yshift=#2cm, rotate=#3, scale=#4]
\SSSFARIV{0}{0}{0}{1}
\SSSFBI{0}{6.85410197}{0}{1}
\SSSFCRI{6.85410197}{0}{0}{1}
\SSSFDIV{6.85410197}{6.85410197}{0}{1}
\end{scope}}

\newcommand{\TAI}[4]{
\begin{scope}[xshift=#1cm, yshift=#2cm, rotate=#3, scale=#4]
\filldraw[color=yellow!50] (0,0) rectangle ({(1+sqrt(5))/2},{(1+sqrt(5))/2});
\end{scope}}

\newcommand{\TAII}[4]{
\begin{scope}[xshift=#1cm, yshift=#2cm, rotate=#3, scale=#4]
\filldraw[color=yellow!50] (0,0) rectangle ({(1+sqrt(5))/2},{(1+sqrt(5))/2});
\end{scope}}

\newcommand{\TAIII}[4]{
\begin{scope}[xshift=#1cm, yshift=#2cm, rotate=#3, scale=#4]
\filldraw[color=yellow!50] (0,0) rectangle ({(1+sqrt(5))/2},{(1+sqrt(5))/2});
\end{scope}}

\newcommand{\TAIV}[4]{
\begin{scope}[xshift=#1cm, yshift=#2cm, rotate=#3, scale=#4]
\filldraw[color=yellow!50] (0,0) rectangle ({(1+sqrt(5))/2},{(1+sqrt(5))/2});;
\end{scope}}

\newcommand{\TARI}[4]{
\begin{scope}[xshift=#1cm, yshift=#2cm, rotate=#3, scale=#4]
\filldraw[color=gray!50] (0,0) rectangle ({(1+sqrt(5))/2},{(1+sqrt(5))/2});
\end{scope}}

\newcommand{\TARII}[4]{
\begin{scope}[xshift=#1cm, yshift=#2cm, rotate=#3, scale=#4]
\filldraw[color=gray!50] (0,0) rectangle ({(1+sqrt(5))/2},{(1+sqrt(5))/2});
\end{scope}}

\newcommand{\TARIII}[4]{
\begin{scope}[xshift=#1cm, yshift=#2cm, rotate=#3, scale=#4]
\filldraw[color=gray!50] (0,0) rectangle ({(1+sqrt(5))/2},{(1+sqrt(5))/2});
\end{scope}}

\newcommand{\TBRI}[4]{
\begin{scope}[xshift=#1cm, yshift=#2cm, rotate=#3, scale=#4]
\filldraw[color=magenta!50] (0,0) rectangle ({(1+sqrt(5))/2},1);
\end{scope}}

\newcommand{\TBII}[4]{
\begin{scope}[xshift=#1cm, yshift=#2cm, rotate=#3, scale=#4]
\filldraw[color=teal!50] (0,0) rectangle ({(1+sqrt(5))/2},1);
\end{scope}}

\newcommand{\TBRII}[4]{
\begin{scope}[xshift=#1cm, yshift=#2cm, rotate=#3, scale=#4]
\filldraw[color=magenta!50] (0,0) rectangle ({(1+sqrt(5))/2},1);
\end{scope}}

\newcommand{\TCI}[4]{
\begin{scope}[xshift=#1cm, yshift=#2cm, rotate=#3, scale=#4]
\filldraw[color=orange!50] (0,0) rectangle (1,{(1+sqrt(5))/2});
\end{scope}}

\newcommand{\TCRI}[4]{
\begin{scope}[xshift=#1cm, yshift=#2cm, rotate=#3, scale=#4]
\filldraw[color=olive!50] (0,0) rectangle (1,{(1+sqrt(5))/2});
\end{scope}}

\newcommand{\TCII}[4]{
\begin{scope}[xshift=#1cm, yshift=#2cm, rotate=#3, scale=#4]
\filldraw[color=orange!50] (0,0) rectangle (1,{(1+sqrt(5))/2});
\end{scope}}

\newcommand{\TCRII}[4]{
\begin{scope}[xshift=#1cm, yshift=#2cm, rotate=#3, scale=#4]
\filldraw[color=olive!50] (0,0) rectangle (1,{(1+sqrt(5))/2});
\end{scope}}

\newcommand{\TDII}[4]{
\begin{scope}[xshift=#1cm, yshift=#2cm, rotate=#3, scale=#4]
\filldraw[color=green!50] (0,0) rectangle (1,1);
\end{scope}}

\newcommand{\TDRI}[4]{
\begin{scope}[xshift=#1cm, yshift=#2cm, rotate=#3, scale=#4]
\filldraw[color=blue!50] (0,0) rectangle (1,1);
\end{scope}}

\newcommand{\TDRII}[4]{
\begin{scope}[xshift=#1cm, yshift=#2cm, rotate=#3, scale=#4]
\filldraw[color=blue!50] (0,0) rectangle (1,1);
\end{scope}}

\newcommand{\TDRIV}[4]{
\begin{scope}[xshift=#1cm, yshift=#2cm, rotate=#3, scale=#4]
\filldraw[color=blue!50] (0,0) rectangle (1,1);
\end{scope}}

\newcommand{\STAIII}[4]{
\begin{scope}[xshift=#1cm, yshift=#2cm, rotate=#3, scale=#4]
\TAII{0}{0}{0}{1}
\TBII{0}{1.61803399}{0}{1}
\TCRII{1.61803399}{0}{0}{1}
\TDRII{1.61803399}{1.61803399}{0}{1}
\end{scope}}

\newcommand{\STARI}[4]{
\begin{scope}[xshift=#1cm, yshift=#2cm, rotate=#3, scale=#4]
\TAIV{0}{0}{0}{1}
\TBRI{0}{1.61803399}{0}{1}
\TCI{1.61803399}{0}{0}{1}
\TDRIV{1.61803399}{1.61803399}{0}{1}
\end{scope}}

\newcommand{\STARIII}[4]{
\begin{scope}[xshift=#1cm, yshift=#2cm, rotate=#3, scale=#4]
\TARII{0}{0}{0}{1}
\TBRII{0}{1.61803399}{0}{1}
\TCII{1.61803399}{0}{0}{1}
\TDII{1.61803399}{1.61803399}{0}{1}
\end{scope}}

\newcommand{\STARIV}[4]{
\begin{scope}[xshift=#1cm, yshift=#2cm, rotate=#3, scale=#4]
\TAI{0}{0}{0}{1}
\TBII{0}{1.61803399}{0}{1}
\TCII{1.61803399}{0}{0}{1}
\TDRI{1.61803399}{1.61803399}{0}{1}
\end{scope}}

\newcommand{\STBI}[4]{
\begin{scope}[xshift=#1cm, yshift=#2cm, rotate=#3, scale=#4]
\TAI{0}{0}{0}{1}
\TCII{1.61803399}{0}{0}{1}
\end{scope}}

\newcommand{\STBRI}[4]{
\begin{scope}[xshift=#1cm, yshift=#2cm, rotate=#3, scale=#4]
\TARI{0}{0}{0}{1}
\TCRII{1.61803399}{0}{0}{1}
\end{scope}}

\newcommand{\STBRII}[4]{
\begin{scope}[xshift=#1cm, yshift=#2cm, rotate=#3, scale=#4]
\TARIII{0}{0}{0}{1}
\TCRI{1.61803399}{0}{0}{1}
\end{scope}}

\newcommand{\STCI}[4]{
\begin{scope}[xshift=#1cm, yshift=#2cm, rotate=#3, scale=#4]
\TAII{0}{0}{0}{1}
\TBII{0}{1.61803399}{0}{1}
\end{scope}}

\newcommand{\STCRI}[4]{
\begin{scope}[xshift=#1cm, yshift=#2cm, rotate=#3, scale=#4]
\TARII{0}{0}{0}{1}
\TBRII{0}{1.61803399}{0}{1}
\end{scope}}

\newcommand{\STCRII}[4]{
\begin{scope}[xshift=#1cm, yshift=#2cm, rotate=#3, scale=#4]
\TAIV{0}{0}{0}{1}
\TBRI{0}{1.61803399}{0}{1}
\end{scope}}

\newcommand{\STDI}[4]{
\begin{scope}[xshift=#1cm, yshift=#2cm, rotate=#3, scale=#4]
\TAI{0}{0}{0}{1}
\end{scope}}

\newcommand{\STDIII}[4]{
\begin{scope}[xshift=#1cm, yshift=#2cm, rotate=#3, scale=#4]
\TAIII{0}{0}{0}{1}
\end{scope}}

\newcommand{\STDIV}[4]{
\begin{scope}[xshift=#1cm, yshift=#2cm, rotate=#3, scale=#4]
\TAIV{0}{0}{0}{1}
\end{scope}}

\newcommand{\SSTAI}[4]{
\begin{scope}[xshift=#1cm, yshift=#2cm, rotate=#3, scale=#4]
\STARIV{0}{0}{0}{1}
\STBI{0}{2.61803399}{0}{1}
\STCRI{2.61803399}{0}{0}{1}
\STDIV{2.61803399}{2.61803399}{0}{1}
\end{scope}}

\newcommand{\SSTAIV}[4]{
\begin{scope}[xshift=#1cm, yshift=#2cm, rotate=#3, scale=#4]
\STARI{0}{0}{0}{1}
\STBRII{0}{2.61803399}{0}{1}
\STCRII{2.61803399}{0}{0}{1}
\STDI{2.61803399}{2.61803399}{0}{1}
\end{scope}}

\newcommand{\SSTARII}[4]{
\begin{scope}[xshift=#1cm, yshift=#2cm, rotate=#3, scale=#4]
\STARIII{0}{0}{0}{1}
\STBRI{0}{2.61803399}{0}{1}
\STCRI{2.61803399}{0}{0}{1}
\STDIII{2.61803399}{2.61803399}{0}{1}
\end{scope}}

\newcommand{\SSTBII}[4]{
\begin{scope}[xshift=#1cm, yshift=#2cm, rotate=#3, scale=#4]
\STAIII{0}{0}{0}{1}
\STCI{2.61803399}{0}{0}{1}
\end{scope}}

\newcommand{\SSTBRII}[4]{
\begin{scope}[xshift=#1cm, yshift=#2cm, rotate=#3, scale=#4]
\STARIII{0}{0}{0}{1}
\STCRI{2.61803399}{0}{0}{1}
\end{scope}}

\newcommand{\SSTCII}[4]{
\begin{scope}[xshift=#1cm, yshift=#2cm, rotate=#3, scale=#4]
\STARIV{0}{0}{0}{1}
\STBI{0}{2.61803399}{0}{1}
\end{scope}}

\newcommand{\SSTDRI}[4]{
\begin{scope}[xshift=#1cm, yshift=#2cm, rotate=#3, scale=#4]
\STARI{0}{0}{0}{1}
\end{scope}}

\newcommand{\SSSTARIV}[4]{
\begin{scope}[xshift=#1cm, yshift=#2cm, rotate=#3, scale=#4]
\SSTAI{0}{0}{0}{1}
\SSTBII{0}{4.23606798}{0}{1}
\SSTCII{4.23606798}{0}{0}{1}
\SSTDRI{4.23606798}{4.23606798}{0}{1}
\end{scope}}

\newcommand{\SSSTBI}[4]{
\begin{scope}[xshift=#1cm, yshift=#2cm, rotate=#3, scale=#4]
\SSTAI{0}{0}{0}{1}
\SSTCII{4.23606798}{0}{0}{1}
\end{scope}}

\newcommand{\SSSTCRI}[4]{
\begin{scope}[xshift=#1cm, yshift=#2cm, rotate=#3, scale=#4]
\SSTARII{0}{0}{0}{1}
\SSTBRII{0}{4.23606798}{0}{1}
\end{scope}}

\newcommand{\SSSTDIV}[4]{
\begin{scope}[xshift=#1cm, yshift=#2cm, rotate=#3, scale=#4]
\SSTAIV{0}{0}{0}{1}
\end{scope}}

\newcommand{\SSSSTAI}[4]{
\begin{scope}[xshift=#1cm, yshift=#2cm, rotate=#3, scale=#4]
\SSSTARIV{0}{0}{0}{1}
\SSSTBI{0}{6.85410197}{0}{1}
\SSSTCRI{6.85410197}{0}{0}{1}
\SSSTDIV{6.85410197}{6.85410197}{0}{1}
\end{scope}}

\newcommand{\hc}[4]{
\begin{scope}[xshift=#1cm, yshift=#2cm, rotate=#3, scale=#4]
\draw[thick] (0,0) rectangle (1,1);
\draw [line width=2pt] (0,0) -- (0.25,0.25)-- (0.25,0.75) -- (0.75,0.75) -- (0.75,0.25) -- (1,0);
\draw[line width=2pt, ->] (0.25,0.75) -- (0.5,0.75);
\end{scope}}

\newcommand{\hcr}[4]{
\begin{scope}[xshift=#1cm, yshift=#2cm, rotate=#3, scale=#4]
\draw[thick] (0,0) rectangle (1,1);
\draw [line width=2pt] (0,0) -- (0.25,0.25)-- (0.25,0.75) -- (0.75,0.75) -- (0.75,0.25) -- (1,0);
\draw[line width=2pt, ->] (0.75,0.75) -- (0.5,0.75);
\end{scope}}

\newcommand{\shc}[4]{
\begin{scope}[xshift=#1cm, yshift=#2cm, rotate=#3, scale=#4]
\hc{0}{1}{0}{1};
\hc{1}{1}{0}{1};
\hcr{0}{1}{-90}{1};
\hcr{2}{0}{90}{1};
\end{scope}}

\newcommand{\shcr}[4]{
\begin{scope}[xshift=#1cm, yshift=#2cm, rotate=#3, scale=#4]
\hcr{0}{1}{0}{1};
\hcr{1}{1}{0}{1};
\hc{0}{1}{-90}{1};
\hc{2}{0}{90}{1};
\end{scope}}

\newcommand{\sshc}[4]{
\begin{scope}[xshift=#1cm, yshift=#2cm, rotate=#3, scale=#4]
\shc{0}{2}{0}{1};
\shc{2}{2}{0}{1};
\shcr{0}{2}{-90}{1};
\shcr{4}{0}{90}{1};
\end{scope}}

\newcommand{\sss}{{\mathrm{supp\ }}}

\usepackage{etoolbox}
\makeatletter
\patchcmd{\@maketitle}
  {\ifx\@empty\@dedicatory}
  {\ifx\@empty\@date \else {\vskip3ex \centering\footnotesize\@date\par\vskip1ex}\fi
   \ifx\@empty\@dedicatory}
  {}{}
\patchcmd{\@adminfootnotes}
  {\ifx\@empty\@date\else \@footnotetext{\@setdate}\fi}
  {}{}{}
\makeatother

\newcommand{\Addresses}{{
  \bigskip
  
  \textsc{MRC Human Genetics Unit, Institute of Genetics and Cancer, University of Edinburgh, Edinburgh, UK}\par\nopagebreak
  \textit{E-mail address:} \texttt{iozkarac@ed.ac.uk}

}}

\tikzset{
    buffer/.style={
        draw,
        shape border rotate=-90,
        isosceles triangle,
        isosceles triangle apex angle=60,
        fill=red,
        node distance=2cm,
        minimum height=4em
    }
}

\title{The Fibonacci Space-Filling Curve}
\author{Mustafa \.{I}smail \"{O}zkaraca}
\date{\today} 

\begin{document}

\maketitle

\begin{abstract}
Can you stretch and reform a curve such that it fills a square completely? This question dates back to 18th century, the origin of space-filling curves. It was proved affirmatively by many great mathematicians. In this document, we reconsider the problem and present a different proof using Cartesian product of Fibonacci substitution with itself. Our construction differs from other curves by design. 
\end{abstract}
\noindent
\section{Introduction}

Can you map a finite line segment onto a square continuously? It is a question that many would probably have not answered affirmatively before late 1800s, the origin of theory of space-filling curves. 

Cantor proved in 1878 that there exists a bijective map between any pair of finite dimensional smooth manifolds \cite{b_cantor}. This result was groundbreaking from many perspectives. For example, a simple consequence of Cantor's result indicates that there is a bijection between an interval (a one-dimensional object) and a square (a two-dimensional object). There was no indication about continuity in Cantor's finding. As such, his finding initiated a hunt to search for continuous bijections between such manifolds, in which existence of such a map would imply a curve, which is a continuous image of an interval, can be fitted into a square bijectively. 

Cantor's result was improved by Netto a year later, where he showed in 1879 that such a map is necessarily discontinuous \cite{b_netto}. Netto's finding directed researchers into looking for continuous surjections. Today's theory of space-filling curves was born by construction of such continuous surjections. 

Peano was the first to settle existence of such a map about a decade later, in 1890 \cite{b_peano}. He constructed a continuous map from the unit interval $[0,1]$ onto the unit square $[0,1]\times[0,1]$. In fact, Peano's mapping was a curve that maps onto the unit square, which is therefore called a \emph{space-filling\ curve}. It is also referred to as \emph{Peano's\ curve} in some sources for acknowledging the initial discoverer Peano. 

Although, Peano's mapping was a curve, the geometry of his curve and how it is filling the square was not apparent since he defined his map analytically, using ternary representations. Indeed, there was no hint of a geometric construction in his proof. The following year, 1891, was the year the theory of space-filling curves stepped into the lands of geometry irrevocably, when Hilbert discovered another such curve based on an elegant geometric construction. 

Hilbert presented a self-similar iterative system referred to as \emph{Hilbert's\ geometric\ construction}, in which a unit square with a curve attached on it is partitioned into four smaller squares, each of which is a scaled replica of the initial square (with the attached curve), and can be concatenated in an ordered fashion, and each smaller squares subsequently partitioned into even smaller squares iteratively by the same token, ad infinitum (Figure \ref{f_Hilbert_SFC_iterations}). This process induces a sequence of curves attached on the unit square that converges to a space-filling curve \cite{sagan}. It should be noted that there are several different equivalent ways to define Hilbert's geometric construction \cite{sagan, b_ozkaraca}. 

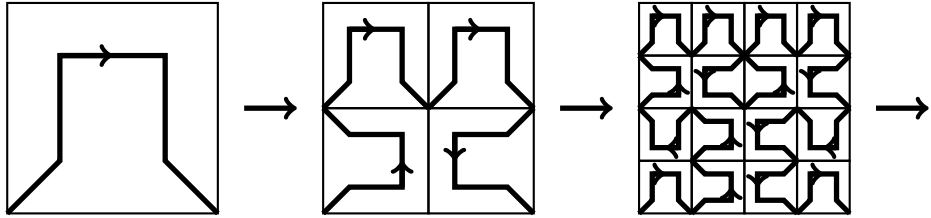
\begin{figure}[H]
\centering
\begin{tikzpicture}[scale=0.7]
\centering
\hc{-9}{0}{0}{4};
\draw [line width=2pt,  ->] (-4.5,2)-- (-3.5,2);
\draw [line width=2pt,  ->] (1.5,2)-- (2.5,2);
\draw [line width=2pt,  ->] (7.5,2)-- (8.5,2);
\shc{-3}{0}{0}{2};
\sshc{3}{0}{0}{1};
\end{tikzpicture}
\caption{A version of Hilbert's geometric construction. A self-similar iterative system induces a sequence of curves that converges to Hilbert's space-filling curve.}
\label{f_Hilbert_SFC_iterations}
\end{figure}

Hilbert's construction answers the question at the beginning affirmatively. In particular, Hilbert's curve can be demonstrated through infinite steps of stretches and reforms of a curve such that it fills the unit square, as follows. Imagine a gum on the shape of the unit interval $[0,1]$ is placed at the bottom of the unit square $[0,1]\times[0,1]$. In the first step, glue the end points of the gum (i.e. $0$ and $1$) to the bottom corners of the square (i.e. $(0,0)$ and $(1,0)$) respectively, and stretch the gum (curve) to make the shape on the leftmost curve in Figure \ref{f_Hilbert_SFC_iterations}. Next, pin three points over the gum such that it is dissected into four equal parts. Glue those three points over the gum to the points $(0,1/2)$, $(1/2,1/2)$ and $(1,1/2)$ over the square, respectively, and stretch every dissected part of the gum to make the shape on the middle curve in the figure. Note that glued points in each step are stationary for the following iterative steps of Hilbert's construction. Continuing stretching and reforming the gum as such leads to Hilbert's curve after infinite steps. Precisely, Hilbert's curve can be represented as a function $H:[0,1]\mapsto[0,1]\times[0,1]$ defined by
\[
H(x) = \lim\limits_{k\to\infty}h_k(x)\quad\mathrm{for\ } x\in[0,1],
\]
where $h_1$ is the curve shown on the left of Figure \ref{f_Hilbert_SFC_iterations}, and $h_k$'s for $k>1$ are curves generated at the iteration steps of Hilbert's construction which are formed by concatenation of curves constructed on the previous step (for example, $h_2$ is formed by concatenation of four curves that are scaled copies of $h_1$, and $h_3$ is formed by concatenation of four curves that are scaled copies of $h_2$ as shown in Figure \ref{f_Hilbert_SFC_iterations}).

More examples were given shortly after Hilbert, based on the same geometric idea of his, such as Moore in 1900 \cite{b_moore}, Sierpinski in 1912 \cite{b_sierpinski} and Polya in 1913 \cite{b_polya}. Those space-filling curve examples were formed by dissecting associated regions with attached curves into smaller replicas so that associated curves can be concatenated in an ordered fashion just like in Hilbert's geometric construction. It is also noteworthy to add that Osgood discovered in 1903 a continuous injection from the unit interval to a subset of the plane whose range has positive Lebesgue measure \cite{b_osgood}. His construction was also based on an iterative geometric construction. These maps are called \emph{Osgood\ curves}.

There are also other important space-filling curves that arise from different techniques, such as Lebesgue's space-filling curve given in 1904 \cite{b_lebesgue} and Schoenberg's space-filling curve given in 1938 \cite{b_schoenberg}. Those curves are constructed based on the idea of interpolation over the map $\pi:\Gamma\mapsto [0,1]\times [0,1]$ defined by 
\begin{equation*}
    \pi\left((0._3\ (2\cdot x_1)(2\cdot x_2)(2\cdot x_3)\dots)\right) = \begin{bmatrix}
(0._2\ x_1 x_3 x_5\dots,)  \\
(0._2\ x_2 x_4 x_6\dots,) 
\end{bmatrix},
\end{equation*}
where $\Gamma$ is the Cantor set and $._3$, $._2$ denote ternary and binary representations, respectively. A detailed geometric analysis of those curves can be found in \cite{Journal_Sagan_1, Journal_Sagan_2, sagan}.

The idea of interpolation over a continuous surjection from the Cantor set became a milestone for the characterisation of space-filling curves. It was proven by Hahn that interpolation approach can be generalised over any given connected locally connected compact subsets of any multi dimensional Euclidean space.  In particular, Hahn showed that any compact connected and locally connected set in any Euclidean space (of any dimension) gives rise to a space-filling curve \cite{b_hahn_3}. Because every space-filling curve is compact connected and locally connected, Hahn characterised space-filling curves in Euclidean spaces of arbitrary dimension. Hahn's idea was inspired from a result in Hausdorff's book \cite{b_hausdorff_felix}, which states that every compact set in an Euclidean space of arbitrary dimension is a continuous image of the Cantor set (i.e. there exists a continuous surjection that maps the Cantor set onto the given compact set). 

This proof was Hahn's second proof for such a characterisation of space-filling curves. Hahn's first proof \cite{b_hahn_1,b_hahn_2} was around the same time when Mazurkiewicz, independently arrived the same conclusion of such characterisation \cite{b_Mazurkiewicz_1, b_Mazurkiewicz_2, b_Mazurkiewicz_3}. The characterisation is therefore usually referred as \emph{Hahn-Mazurkiewicz\ Theorem} in the literature. A proof of this characterisation using the idea of Hahn's second proof, can be found in \cite[Theorem 6.8]{sagan} or in \cite[Theorem 31.5]{b_willard_top_book}. 
 
Hahn's program can be regarded as an algorithm to construct a space-filling curve. It is based on an interpolation argument. There are also other important methods which can be regarded as algorithms for constructing space-filling curves, such as the L-system method introduced by Lindenmayer \cite{b_lindenmayer}, the recurrent set method defined by Dekking \cite{b_dekking}, the edge-to-trail substitution method developed by X. R. Dai, H. Rao and S. Q. Zhang in subsequent papers \cite{b_rao_zang_1, b_rao_zang_2, b_rao_zang_3} and a generalisation of Lebesgue's interpolation method defined by {\"O}zkaraca \cite{ozkaraca_Lebesgue}. The L-system method is an algorithm to construct self-similar sets by a set of rules that are formed at the first step, together with an alphabet, in which the rules are going to be applied. The recurrent set method can be thought of as an improvement to the L-system method. These algorithms start with a (one dimensional) substitution structure. They define a space-filling curves alike structure by iterating the substitutive structure ad infinitum. The edge-to-trail substitution method, on the other hand, generates space filing curves for self-similar sets satisfying mild conditions. Space-filling curves are being formed by constructing a substitution rule via the geometric structure of given self-similar sets. Lastly, {\"O}zkaraca's method generalises interpolation argument of Lebesgue's for finite substitutions satisfying mild conditions. In this paper, we generate a space-filling curve using Fibonacci substitution with the help of Hilbert's geometric construction.

\section{Substitutions}
In this section we borrow some terminology from theory of tiling spaces.

\subsection{Substitution Rules}
Most of the material in this section is classical and adopted from textbooks such as \cite{sadun2008topology}. 

Let $\RR^n$ denote the n-dimensional Euclidean space for $n\in\ZZ^+$. We define the following:
\begin{enumerate}
    \item A \emph{tile} $t$ is a subset of $\RR^n$ which is homeomorphic to the closed unit ball. We further add a label $l(t)$ to $t$ that differentiates $t$ from any other identical sets. Examples of commonly used labels are letters, colors or both.
    
    The associated set of $t$ is called the \emph{support} of $t$ and is denoted by $\sss t$. An example of a tile in $\RR^2$ is given in Figure \ref{Figure_Basic_Substitution_Definitions_1}. 
    \item A \emph{patch} $P$ (in $\RR^n$) is a finite collection of tiles (in $\RR^n$) so that
\begin{enumerate}
\item[(i)] $\bigcup\limits_{p\in P}\sss p$ is homeomorphic to the closed unit ball,
\item[(ii)] $\mathrm{int\ }(\sss p)\cap\mathrm{int\ }(\sss q)=\emptyset$ for every distinct pair $p,q\in P$. 
\end{enumerate}
The set $\bigcup\limits_{p\in P}\sss p$ is called the \emph{support\ of} $P$. That is, support of a patch is the union of supports of tiles it consists of. An example of a patch in $\RR^2$ is given in Figure \ref{Figure_Basic_Substitution_Definitions_2}. 
\item Scaled and/or translated copies of a tile (or a patch) is also a tile (or a patch). In particular, for a given tile $t$, a patch $P$, a vector $x\in\RR^n$ and a non-zero scalar $\lambda\neq0$, we create new tiles $t+x$, $\lambda t$ and new patches $P+x$, $\lambda P$ by the following relations:
\begin{itemize}
    \item[(i)] $\sss (t+x)=(\sss t) + x$ and $l(t+x)=l(t)$. The tile $t+x$ is called a \emph{translation} of $t$.
    \item[(ii)] $\sss (\lambda t)=\lambda \sss t$ and $l(\lambda t)=l(t)$. The tile $\lambda t$ is called a \emph{scaled\ copy} of $t$.
    \item[(iii)] $P+x=\{p+x:\ p\in P\}$. The patch $P+x$ is called a \emph{translation} of $P$.
    \item[(iv)] $\lambda P=\{\lambda p:\ p\in P\}$. The patch $\lambda P$ is called a \emph{scaled\ copy} of $P$.   
\end{itemize}
Examples of scaled and translated copies of a tile and a patch in $\RR^2$ is given in Figure \ref{Figure_Basic_Substitution_Definitions_3}.
\item Let $t$ be a given tile. A simple curve $e_t$ with end points $x,y$ for some distinct $x,y\in V_t$, where $V_t$ is the collection of vertices of $\sss t$, is called a \emph{decoration} for $t$. For a given tile $t$ and a decoration $e_t$ for $t$, we create a new tile $t^\prime$ by relations $\sss t^\prime=\sss t$ and $l(t^\prime)=(l(t), e_t)$. The constructed tile $t^\prime$ is called a \emph{decorated\ tile} (of $t$). An example of a decorated tile is given in Figure \ref{Figure_Basic_Substitution_Definitions_4}.
\end{enumerate}

\begin{figure}[ht]
\centering
\begin{subfigure}[b]{0.45\textwidth}
\centering
\begin{tikzpicture}
\draw[dashed] (0,-1) -- (0,2);
\draw[dashed] (-1,0) -- (2,0);
\draw (0,0) rectangle (1,1);
\filldraw[black] (0,0) circle (2pt) node[anchor=north east] {$(0,0)$};
\filldraw[black] (1,0) circle (2pt) node[anchor=north west] {$(1,0)$};
\filldraw[black] (0,1) circle (2pt) node[anchor=south east] {$(0,1)$};
\node at (0.5, 0.5) {$a$};
\end{tikzpicture} 
\caption{A tile $t$ with $\sss t=[0,1]\times [0,1]$ and $l(t)=a$.}
\label{Figure_Basic_Substitution_Definitions_1}
\end{subfigure}
\hfill
\begin{subfigure}[b]{0.45\textwidth}
\centering
\begin{tikzpicture}
\draw[dashed] (0,-1) -- (0,3);
\draw[dashed] (-1,0) -- (3,0);
\draw[step=1] (0,0) grid (2,2);
\filldraw[black] (0,0) circle (2pt) node[anchor=north east] {$(0,0)$};
\filldraw[black] (1,0) circle (2pt) node[anchor=north] {$(1,0)$};
\filldraw[black] (2,0) circle (2pt) node[anchor=north] {$(2,0)$};
\filldraw[black] (0,1) circle (2pt) node[anchor=east] {$(0,1)$};
\filldraw[black] (0,2) circle (2pt) node[anchor=east] {$(0,2)$};
\node at (0.5, 0.5) {$a$};
\node at (1.5, 0.5) {$b$};
\node at (0.5, 1.5) {$c$};
\node at (1.5, 1.5) {$d$};
\end{tikzpicture} 
\caption{A patch $P$ with $\sss P=[0,2]\times [0,2]$. It consists of $4$ square tiles with labels $a,b,c,d$.}
\label{Figure_Basic_Substitution_Definitions_2}
\end{subfigure}
\vskip\baselineskip
\begin{subfigure}[b]{0.45\textwidth}
\centering
\begin{tikzpicture}
\draw[dashed] (0,-1) -- (0,3);
\draw[dashed] (-3,0) -- (4,0);
\draw (-1,0) rectangle (-2,1);
\filldraw[black] (0,0) circle (2pt);
\filldraw[black] (-1,0) circle (2pt) node[anchor=north] {$(-2,0)$};
\filldraw[black] (-2,0) circle (2pt) node[anchor=north east] {$(-4,0)$};
\filldraw[black] (-2,1) circle (2pt) node[anchor=south east] {$(-4,2)$};
\filldraw[black] (1,0) circle (2pt) node[anchor=north] {$(2,0)$};
\filldraw[black] (2,0) circle (2pt) node[anchor=north] {$(4,0)$};
\filldraw[black] (3,0) circle (2pt) node[anchor=north] {$(6,0)$};
\filldraw[black] (1,1) circle (2pt) node[anchor=east] {$(2,2)$};
\filldraw[black] (1,2) circle (2pt) node[anchor=east] {$(2,4)$};
\node at (-1.5, 0.5) {$a$};

\draw[step=1] (1,0) grid (3,2);
\node at (1.5, 0.5) {$a$};
\node at (2.5, 0.5) {$b$};
\node at (1.5, 1.5) {$c$};
\node at (2.5, 1.5) {$d$};
\end{tikzpicture}
\caption{Tile $2t+(-4,0)$ and patch $2P+(2,0)$ for the tile $t$ and patch $P$ in Figure \ref{Figure_Basic_Substitution_Definitions_1} and Figure \ref{Figure_Basic_Substitution_Definitions_2} respectively. Figure is scaled for demonstration purposes.}
\label{Figure_Basic_Substitution_Definitions_3}
\end{subfigure}
\hfill
\begin{subfigure}[b]{0.45\textwidth}
\centering
\begin{tikzpicture}
\draw[dashed] (0,-1) -- (0,3);
\draw[dashed] (-1,0) -- (3,0);
\draw[step=1] (0,0) rectangle (2,2);
\filldraw[black] (0,0) circle (2pt) node[anchor=north east] {$(0,0)$};
\filldraw[black] (2,0) circle (2pt) node[anchor=north] {$(1,0)$};
\filldraw[black] (0,2) circle (2pt) node[anchor=east] {$(0,1)$};
\node at (1, 1.7) {$a$};
\draw[->, very thick] (0,0) -- (1,1);
\draw[very thick] (2,0) -- (1,1);

\end{tikzpicture}
\caption{A decorated tile $t^\prime$ (of $t$) for the tile $t$ in Figure \ref{Figure_Basic_Substitution_Definitions_1}. It has label $l(t^\prime)=(a,e_t)$ where $e_t$ is the attached curve with end points $(0,0)$ and $(1,0)$. Figure is scaled for demonstration purposes.}
\label{Figure_Basic_Substitution_Definitions_4}
\end{subfigure}
\caption{Tiling terminologies}
\label{Figure_Basic_Substitution_Definitions}
\end{figure}
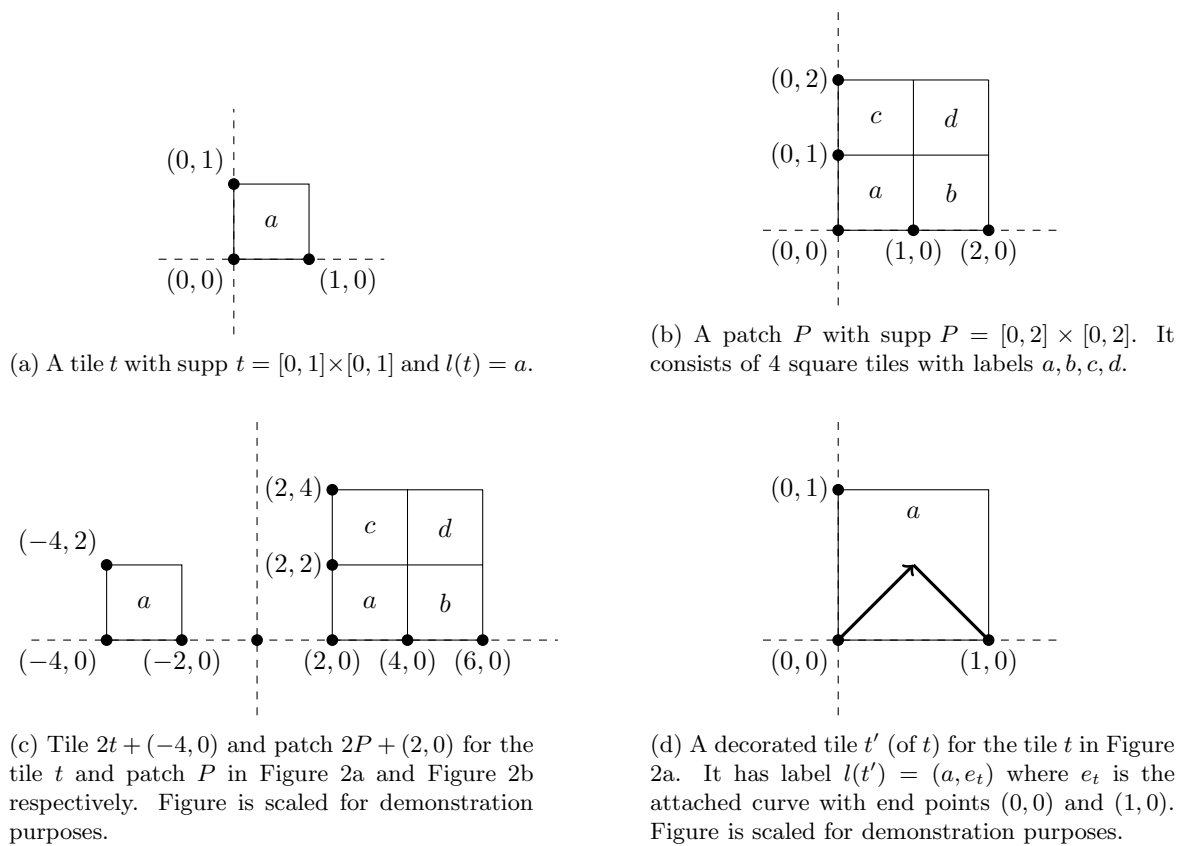

\begin{definition} 
Let $\cP$ be a given finite collection of tiles (in $\RR^n$) and $\cP^*$ denote the set of all patches which are collection of tiles, each of which is a translation of a tile in $\cP$. A function $\mu:\cP\mapsto\cP^*$ is called a \emph{substitution} (in $\RR^n$) if there exists a scalar $\lambda>1$ such that $\sss \mu(p)=\lambda\cdot\sss \mu(p)$ for each $p\in\cP$.

The scalar $\lambda$ is called the \emph{expansion\ factor} of $\mu$.
\end{definition}

\begin{remark}[Powers of substitutions]
Note that, in general, domain and range of a substitution are mutually exclusive. That is, $\mu^2=\mu\circ\mu$ may be ill-defined. In order to address this issue we extend a given substitution $\mu:\cP\mapsto\cP^*$ to $\mu^\prime:\cP+\RR^n\mapsto\cP^*$ by the relation
\[
\mu^\prime(p+x)=\mu(p)+\lambda\cdot x\quad\text{for}\ p\in\cP\ \text{and}\ x\in\RR^n.
\]
As such, we define $(\mu^\prime)^k(p)$ for $p\in\cP$ and $k\geq 2$ recursively as follows:
\[
{\displaystyle (\mu^\prime)^k(p):=\bigcup\limits_{t\in(\mu^\prime)^{k-1}(p)}\mu^\prime(t).}
\]
Observe that $(\mu^\prime)^k$ for each $k\in\ZZ^+$ is well-defined. For simplicity, we refer to $(\mu^\prime)^k$ as $\mu^k$ throughout the document.
\end{remark}

\begin{definition}
Let $\mu:\cP\mapsto\cP^*$ be a given substitution. The patch $\mu^k(p)$ for $p\in\cP$ and $k\in\NN$ is called a \emph{k-supertile} of $\mu$, with the convention that $\mu^0(p):=p$.
\end{definition}

\begin{example}[Fibonacci Substitution]
Consider the substitution $\mu_1$ (in $\RR$) shown in Figure \ref{Figure_Fibonacci_Substitution}. It is defined over two intervals $[0,\phi]$ and $[0,1]$ with labels $A,B$, respectively, where $\phi = (1+\sqrt{5})/2$ is the golden mean. The expansion factor for $\mu_1$ is $\phi$. The substitution $\mu_1$ is called the \emph{Fibonacci\ substitution}.
\end{example}
\begin{example}[Cartesian Product of Fibonacci Substitutions]
The substitution $\mu_2$ (in $\RR^2$) shown in Figure \ref{Figure_Fibonacci_times_Fibonacci_Substitution} is generated by taking the Cartesian product of two Fibonacci substitutions. It is defined over four rectangle tiles with labels $A,B,C,D$. The expansion factor for $\mu_2$ is also $\phi$. The substitution $\mu_2$ is called the \emph{Cartesian\ product\ of\ Fibonacci\ substitutions}.
\end{example}

\begin{figure}[ht]
\centering
\begin{tikzpicture}
\draw (0,0) -- (1.6180339887498948482045868343656381177203091798057628621354486227,0);
\draw[very thick, ->] (2,0) -- (2.72,0);
\node at (2.36,0.25) {$\mu_1$};

\draw (3,0) -- (5.6180339887498948482045868343656381177203091798057628621354486227,0);
\draw (0,-0.2) -- (0,0.2);
\draw (1.6180339887498948482045868343656381177203091798057628621354486227,-0.2) -- (1.6180339887498948482045868343656381177203091798057628621354486227,0.2);
\draw (3,-0.2) -- (3,0.2);
\draw (5.6180339887498948482045868343656381177203091798057628621354486227,0.2) -- (5.6180339887498948482045868343656381177203091798057628621354486227,-0.2);
\draw (4.6180339887498948482045868343656381177203091798057628621354486227,0.2) -- (4.6180339887498948482045868343656381177203091798057628621354486227,-0.2);

\node[anchor=north] at (0,-0.2) {$0$};
\node[anchor=north] at (1.6180339887498948482045868343656381177203091798057628621354486227,-0.1) 
{$\phi$};

\node[anchor=north] at (3,-0.2) {$0$};
\node[anchor=north] at (4.6180339887498948482045868343656381177203091798057628621354486227,-0.1) 
{$\phi$};
\node[anchor=north] at (5.6180339887498948482045868343656381177203091798057628621354486227,-0.1) 
{$1+\phi$};

\node at (0.8090169943749474241022934171828190588601545899028814310677243113,0.3) {$A$};
\node at (3.8090169943749474241022934171828190588601545899028814310677243113,0.3) {$A$};
\node at (0.5+4.6180339887498948482045868343656381177203091798057628621354486227,0.3) {$B$};

\draw (7+2,0) -- (8+2,0);
\draw[very thick, ->] (8.25+2,0) -- (8.75+2,0);
\node at (8.5+2,0.25) {$\mu_1$};

\draw (9+2,0) -- (10.6180339887498948482045868343656381177203091798057628621354486227+2,0);
\draw (9+2,-0.2) -- (9+2,0.2);
\draw (10.6180339887498948482045868343656381177203091798057628621354486227+2,0.2) --(10.6180339887498948482045868343656381177203091798057628621354486227+2,-0.2);
\draw (7+2,0.2) -- (7+2,-0.2);
\draw (8+2,0.2) -- (8+2,-0.2);

\node[anchor=north] at (7+2,-0.2) {$0$};
\node[anchor=north] at (8+2,-0.2) {$1$};
\node at (7.5+2,0.3) {$B$};
\node at (9.8090169943749474241022934171828190588601545899028814310677243113+2,0.3) {$A$};

\node[anchor=north] at (9+2,-0.2) {$0$};
\node[anchor=north] at (10.6180339887498948482045868343656381177203091798057628621354486227+2,-0.1) 
{$\phi$};

\end{tikzpicture}    
\caption{The Fibonacci substitution $\mu_1$.}
\label{Figure_Fibonacci_Substitution}
\end{figure}
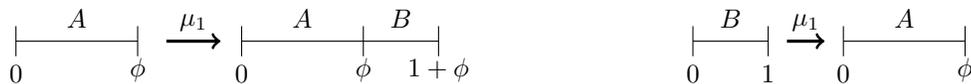

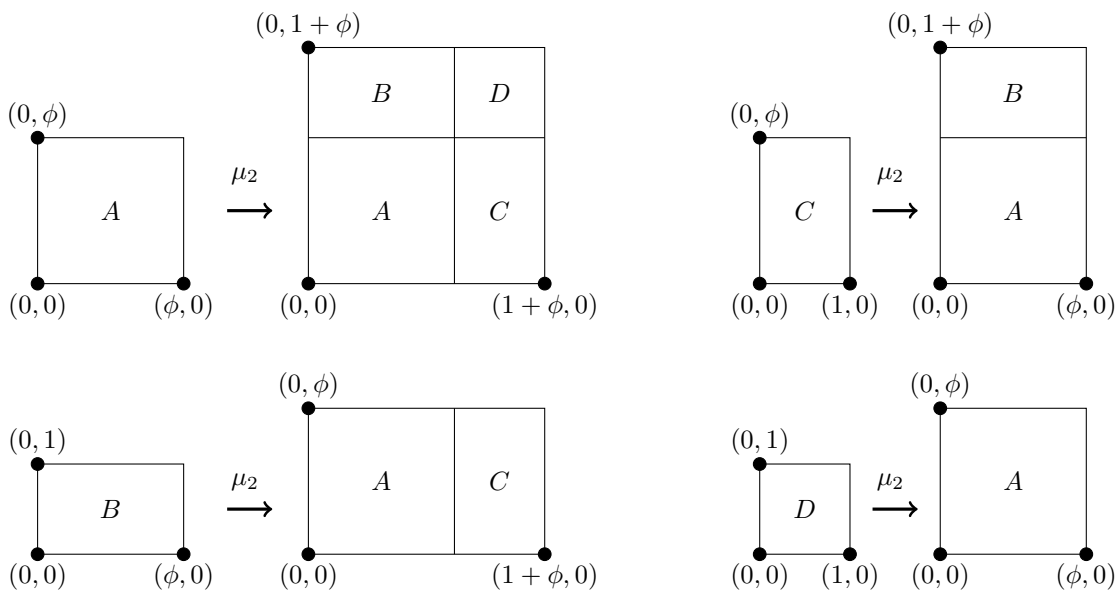
\begin{figure}[ht]
\centering
\begin{tikzpicture}[scale=1.2]
\draw (0,0) rectangle (1.6180339887498948482045868343656381177203091798057628621354486227,1.6180339887498948482045868343656381177203091798057628621354486227);
\draw[very thick, ->] (2.1,0.8090169943749474241022934171828190588601545899028814310677243113) -- (2.6,0.8090169943749474241022934171828190588601545899028814310677243113); 

\node at (2.3,1.2) {$\mu_2$};

\draw (3,0) rectangle (5.6180339887498948482045868343656381177203091798057628621354486227,2.6180339887498948482045868343656381177203091798057628621354486227);
\draw (8,0) rectangle (9,1.6180339887498948482045868343656381177203091798057628621354486227);
\draw[very thick, ->] (9.25,0.8090169943749474241022934171828190588601545899028814310677243113) -- (9.75,0.8090169943749474241022934171828190588601545899028814310677243113);
\node at (9.45,1.2) {$\mu_2$};

\draw (10,0) rectangle (11.6180339887498948482045868343656381177203091798057628621354486227,2.6180339887498948482045868343656381177203091798057628621354486227);
\draw (10,1.6180339887498948482045868343656381177203091798057628621354486227) -- (11.6180339887498948482045868343656381177203091798057628621354486227,1.6180339887498948482045868343656381177203091798057628621354486227);
\draw (3,1.6180339887498948482045868343656381177203091798057628621354486227) -- (5.6180339887498948482045868343656381177203091798057628621354486227,1.6180339887498948482045868343656381177203091798057628621354486227);
\draw (4.6180339887498948482045868343656381177203091798057628621354486227,2.6180339887498948482045868343656381177203091798057628621354486227) -- (4.6180339887498948482045868343656381177203091798057628621354486227,0);
\draw (0,-3) rectangle (1.6180339887498948482045868343656381177203091798057628621354486227,1-3);
\draw[very thick, ->] (2.1,0.5-3) -- (2.6,0.5-3);
\node at (2.3,-2.2) {$\mu_2$};
\draw (3,-3) rectangle (5.6180339887498948482045868343656381177203091798057628621354486227,1.6180339887498948482045868343656381177203091798057628621354486227-3);
\draw (4.6180339887498948482045868343656381177203091798057628621354486227,1.6180339887498948482045868343656381177203091798057628621354486227-3) -- (4.6180339887498948482045868343656381177203091798057628621354486227,1.6180339887498948482045868343656381177203091798057628621354486227-3-1.6180339887498948482045868343656381177203091798057628621354486227);
\draw (8,-3) rectangle (9,1-3);
\draw[very thick, ->] (9.25,0.5-3) -- (9.75,0.5-3);
\node at (9.45,-2.2) {$\mu_2$};
\draw (10,0-3) rectangle (11.6180339887498948482045868343656381177203091798057628621354486227,1.6180339887498948482045868343656381177203091798057628621354486227-3);
\node at (0.8090169943749474241022934171828190588601545899028814310677243113,0.8090169943749474241022934171828190588601545899028814310677243113) {$A$};
\node at (3.8090169943749474241022934171828190588601545899028814310677243113,0.8090169943749474241022934171828190588601545899028814310677243113) {$A$};
\node at (3.8090169943749474241022934171828190588601545899028814310677243113,0.8090169943749474241022934171828190588601545899028814310677243113-3) {$A$};
\node at (3.8090169943749474241022934171828190588601545899028814310677243113+1.3090169943749474241022934171828190588601545899028814310677243113,0.8090169943749474241022934171828190588601545899028814310677243113-3) {$C$};
\node at (3.8090169943749474241022934171828190588601545899028814310677243113+1.3090169943749474241022934171828190588601545899028814310677243113,0.8090169943749474241022934171828190588601545899028814310677243113) {$C$};
\node at (3.8090169943749474241022934171828190588601545899028814310677243113+1.3090169943749474241022934171828190588601545899028814310677243113,0.8090169943749474241022934171828190588601545899028814310677243113+1.3090169943749474241022934171828190588601545899028814310677243113) {$D$};
\node at (3.8090169943749474241022934171828190588601545899028814310677243113,0.8090169943749474241022934171828190588601545899028814310677243113+1.3090169943749474241022934171828190588601545899028814310677243113) {$B$};
\node at (10.8090169943749474241022934171828190588601545899028814310677243113,0.8090169943749474241022934171828190588601545899028814310677243113+1.3090169943749474241022934171828190588601545899028814310677243113) {$B$};
\node at (8.5,0.8090169943749474241022934171828190588601545899028814310677243113) {$C$};
\node at (8.5,0.5-3) {$D$};
\node at (0.8090169943749474241022934171828190588601545899028814310677243113,0.5-3) {$B$};
\node at (10.8090169943749474241022934171828190588601545899028814310677243113,0.8090169943749474241022934171828190588601545899028814310677243113-3) {$A$};
\node at (10.8090169943749474241022934171828190588601545899028814310677243113,0.8090169943749474241022934171828190588601545899028814310677243113) {$A$};
\filldraw[black] (0,0) circle (2pt) node[anchor=north] {$(0,0)$};
\filldraw[black] (3,0) circle (2pt) node[anchor=north] {$(0,0)$} ; 
\filldraw[black] (0,1.6180339887498948482045868343656381177203091798057628621354486227) circle (2pt) node[anchor=south] 
{$(0,\phi)$};
\filldraw[black] (3,2.6180339887498948482045868343656381177203091798057628621354486227) circle (2pt) node[anchor=south] 
{$(0,1+\phi)$};
\filldraw[black] (0,-3) circle (2pt) node[anchor=north] {$(0,0)$}; 
\filldraw[black] (3,-3) circle (2pt) node[anchor=north] {$(0,0)$}; 
\filldraw[black] (0,1-3) circle (2pt) node[anchor=south] {$(0,1)$};
\filldraw[black] (3,1.6180339887498948482045868343656381177203091798057628621354486227-3) circle (2pt) node[anchor=south] 
{$(0,\phi)$};

\filldraw[black] (3+2.6180339887498948482045868343656381177203091798057628621354486227,0-3) circle (2pt) node[anchor=north] 
{$(1+\phi,0)$};
\filldraw[black] (1.6180339887498948482045868343656381177203091798057628621354486227,0-3) circle (2pt) node[anchor=north] 
{$(\phi,0)$};
\filldraw[black] (1.6180339887498948482045868343656381177203091798057628621354486227,0) circle (2pt) node[anchor=north] 
{$(\phi,0)$};
\filldraw[black] (5.6180339887498948482045868343656381177203091798057628621354486227,0) circle (2pt) node[anchor=north] 
{$(1+\phi,0)$};
\filldraw[black] (8,0) circle (2pt) node[anchor=north] {$(0,0)$}; 
\filldraw[black] (9,0) circle (2pt) node[anchor=north] {$(1,0)$}; 
\filldraw[black] (8,0-3) circle (2pt) node[anchor=north] {$(0,0)$}; 
\filldraw[black] (9,0-3) circle (2pt) node[anchor=north] {$(1,0)$}; 
\filldraw[black] (8,1.6180339887498948482045868343656381177203091798057628621354486227) circle (2pt) node[anchor=south] 
{$(0,\phi)$};
\filldraw[black] (10,0) circle (2pt) node[anchor=north] {$(0,0)$}; 
\filldraw[black]  (10,-3) circle (2pt) node[anchor=north] {$(0,0)$}; 
\filldraw[black]  (10,2.6180339887498948482045868343656381177203091798057628621354486227) circle (2pt) node[anchor=south] 
{$(0,1+\phi)$};
\filldraw[black] (10,1.6180339887498948482045868343656381177203091798057628621354486227-3) circle (2pt) node[anchor=south]  
{$(0,\phi)$};
\filldraw[black] (10+1.6180339887498948482045868343656381177203091798057628621354486227,0) circle (2pt) node[anchor=north]  
{$(\phi,0)$}; 
\filldraw[black] (10+1.6180339887498948482045868343656381177203091798057628621354486227,0-3) circle (2pt) node[anchor=north]  
{$(\phi,0)$}; 
\filldraw[black] (8,0-3+1) circle (2pt) node[anchor=south]  {$(0,1)$}; 
\end{tikzpicture}    
\caption{The Cartesian product of Fibonacci Substitutions.}
\label{Figure_Fibonacci_times_Fibonacci_Substitution}
\end{figure}

\subsection{A Substitution over Decorated Tiles}

Consider the collection of 24 rectangle tiles $\{p_1,\dots,p_{24}\}$ shown in Figure \ref{F_fibonacci_substitution_prototiles}, which are defined by the following relations:
\begin{align*}
\sss(p_1) = [0,\phi]\times[0,\phi] & \quad\quad\text{and}\quad\quad l(p_1) = A^+_1, \\
\sss(p_2) = [0,\phi]\times[0,\phi] & \quad\quad\text{and}\quad\quad l(p_2) = A^+_2, \\
\sss(p_3) = [0,\phi]\times[0,\phi] & \quad\quad\text{and}\quad\quad l(p_3) = A^+_3, \\
\sss(p_4) = [0,\phi]\times[0,\phi] & \quad\quad\text{and}\quad\quad l(p_4) = A^+_4, \\
\sss(p_5) = [0,\phi]\times[0,1] & \quad\quad\text{and}\quad\quad l(p_5) = B^+_1, \\
\sss(p_6) = [0,\phi]\times[0,1] & \quad\quad\text{and}\quad\quad l(p_6) = B^+_2, \\
\sss(p_7) = [0,1]\times[0,\phi] & \quad\quad\text{and}\quad\quad l(p_7) = C^+_1, \\
\sss(p_8) = [0,1]\times[0,\phi] & \quad\quad\text{and}\quad\quad l(p_8) = C^+_2, \\
\sss(p_9) = [0,1]\times[0,1] & \quad\quad\text{and}\quad\quad l(p_9) = D^+_1, \\
\sss(p_{10}) = [0,1]\times[0,1] & \quad\quad\text{and}\quad\quad l(p_{10}) = D^+_2, \\
\sss(p_{11}) = [0,1]\times[0,1] & \quad\quad\text{and}\quad\quad l(p_{11}) = D^+_3, \\
\sss(p_{12}) = [0,1]\times[0,1] & \quad\quad\text{and}\quad\quad l(p_{12}) = D^+_4, \\   
\sss(p_{13}) = [0,\phi]\times[0,\phi] & \quad\quad\text{and}\quad\quad l(p_{13}) = A^-_1, \\
\sss(p_{14}) = [0,\phi]\times[0,\phi] & \quad\quad\text{and}\quad\quad l(p_{14}) = A^-_2, \\
\sss(p_{15}) = [0,\phi]\times[0,\phi] & \quad\quad\text{and}\quad\quad l(p_{15}) = A^-_3, \\
\sss(p_{16}) = [0,\phi]\times[0,\phi] & \quad\quad\text{and}\quad\quad l(p_{16}) = A^-_4, \\
\sss(p_{17}) = [0,\phi]\times[0,1] & \quad\quad\text{and}\quad\quad l(p_{17}) = B^-_1, \\
\sss(p_{18}) = [0,\phi]\times[0,1] & \quad\quad\text{and}\quad\quad l(p_{18}) = B^-_2, \\
\sss(p_{19}) = [0,1]\times[0,\phi] & \quad\quad\text{and}\quad\quad l(p_{19}) = C^-_1, \\
\sss(p_{20}) = [0,1]\times[0,\phi] & \quad\quad\text{and}\quad\quad l(p_{20}) = C^-_2, \\
\sss(p_{21}) = [0,1]\times[0,1] & \quad\quad\text{and}\quad\quad l(p_{21}) = D^-_1, \\
\sss(p_{22}) = [0,1]\times[0,1] & \quad\quad\text{and}\quad\quad l(p_{22}) = D^-_2, \\
\sss(p_{23}) = [0,1]\times[0,1] & \quad\quad\text{and}\quad\quad l(p_{23}) = D^-_3, \\
\sss(p_{24}) = [0,1]\times[0,1] & \quad\quad\text{and}\quad\quad l(p_{24}) = D^-_4,
\end{align*}
where $\phi=(1+\sqrt{5})/2$. Each rectangle is distinguished by a color $c\in\{A,B,C,D\}$ and an attached oriented curve (decoration) $e_i^j$ for $j\in\{-,+\}$ and $i\in\{1,2,3,4\}$. This information is summarised in labels. For example $l(p_1)=A^+_1$ indicates that $p_1$ has a color $A$ and an attached curve $e_1^+$ which has a start point $(0,0)$ and an end point $(1,0)$. 

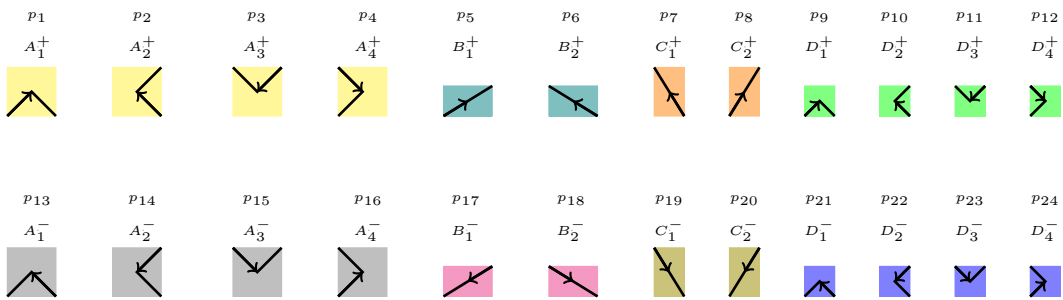
\begin{figure}[H]
\centering
\begin{tikzpicture}[scale=0.4]
\centering

\filldraw[white] (35,0) circle (0pt);

\filldraw[black] (0.5,2.8-0.5) circle (0pt) node{\tiny{$A^+_1$}};
\filldraw[black] (4,2.8-0.5) circle (0pt) node{\tiny{$A^+_2$}};
\filldraw[black] (7.8,2.8-0.5) circle (0pt) node{\tiny{$A^+_3$}};
\filldraw[black] (11.5,2.8-0.5) circle (0pt) node{\tiny{$A^+_4$}};
\filldraw[black] (14.75,2.8-0.5) circle (0pt) node{\tiny{$B^+_1$}};
\filldraw[black] (18.25,2.8-0.5) circle (0pt) node{\tiny{$B^+_2$}};
\filldraw[black] (21.5,2.8-0.5) circle (0pt) node{\tiny{$C^+_1$}};
\filldraw[black] (24,2.8-0.5) circle (0pt) node{\tiny{$C^+_2$}};
\filldraw[black] (26.5,2.8-0.5) circle (0pt) node{\tiny{$D^+_1$}};
\filldraw[black] (29,2.8-0.5) circle (0pt) node{\tiny{$D^+_2$}};
\filldraw[black] (31.5,2.8-0.5) circle (0pt) node{\tiny{$D^+_3$}};
\filldraw[black] (34,2.8-0.5) circle (0pt) node{\tiny{$D^+_4$}};

\filldraw[black] (0.5,3.8-0.5) circle (0pt) node{\tiny{$p_1$}};
\filldraw[black] (4,3.8-0.5) circle (0pt) node{\tiny{$p_2$}};
\filldraw[black] (7.8,3.8-0.5) circle (0pt) node{\tiny{$p_3$}};
\filldraw[black] (11.5,3.8-0.5) circle (0pt) node{\tiny{$p_4$}};
\filldraw[black] (14.75,3.8-0.5) circle (0pt) node{\tiny{$p_5$}};
\filldraw[black] (18.25,3.8-0.5) circle (0pt) node{\tiny{$p_6$}};
\filldraw[black] (21.5,3.8-0.5) circle (0pt) node{\tiny{$p_7$}};
\filldraw[black] (24,3.8-0.5) circle (0pt) node{\tiny{$p_8$}};
\filldraw[black] (26.5,3.8-0.5) circle (0pt) node{\tiny{$p_9$}};
\filldraw[black] (29,3.8-0.5) circle (0pt) node{\tiny{$p_{10}$}};
\filldraw[black] (31.5,3.8-0.5) circle (0pt) node{\tiny{$p_{11}$}};
\filldraw[black] (34,3.8-0.5) circle (0pt) node{\tiny{$p_{12}$}};

\FAI{-0.5}{0}{0}{1};
\FAII{3}{0}{0}{1};
\FAIII{7}{0}{0}{1};
\FAIV{10.5}{0}{0}{1};
\FBI{14}{0}{0}{1};
\FBII{17.5}{0}{0}{1};
\FCI{21}{0}{0}{1};
\FCII{23.5}{0}{0}{1};
\FDI{26}{0}{0}{1};
\FDII{28.5}{0}{0}{1};
\FDIII{31}{0}{0}{1};
\FDIV{33.5}{0}{0}{1};

\FARI{-0.5}{-6}{0}{1};
\FARII{3}{-6}{0}{1};
\FARIII{7}{-6}{0}{1};
\FARIV{10.5}{-6}{0}{1};
\FBRI{14}{-6}{0}{1};
\FBRII{17.5}{-6}{0}{1};
\FCRI{21}{-6}{0}{1};
\FCRII{23.5}{-6}{0}{1};
\FDRI{26}{-6}{0}{1};
\FDRII{28.5}{-6}{0}{1};
\FDRIII{31}{-6}{0}{1};
\FDRIV{33.5}{-6}{0}{1};

\filldraw[black] (0.5,-3.8) circle (0pt) node{\tiny{$A^-_1$}};
\filldraw[black] (4,-3.8) circle (0pt) node{\tiny{$A^-_2$}};
\filldraw[black] (7.8,-3.8) circle (0pt) node{\tiny{$A^-_3$}};
\filldraw[black] (11.5,-3.8) circle (0pt) node{\tiny{$A^-_4$}};
\filldraw[black] (14.75,-3.8) circle (0pt) node{\tiny{$B^-_1$}};
\filldraw[black] (18.25,-3.8) circle (0pt) node{\tiny{$B^-_2$}};
\filldraw[black] (21.5,-3.8) circle (0pt) node{\tiny{$C^-_1$}};
\filldraw[black] (24,-3.8) circle (0pt) node{\tiny{$C^-_2$}};
\filldraw[black] (26.5,-3.8) circle (0pt) node{\tiny{$D^-_1$}};
\filldraw[black] (29,-3.8) circle (0pt) node{\tiny{$D^-_2$}};
\filldraw[black] (31.5,-3.8) circle (0pt) node{\tiny{$D^-_3$}};
\filldraw[black] (34,-3.8) circle (0pt) node{\tiny{$D^-_4$}};

\filldraw[black] (0.5,-2.8) circle (0pt) node{\tiny{$p_{13}$}};
\filldraw[black] (4,-2.8) circle (0pt) node{\tiny{$p_{14}$}};
\filldraw[black] (7.8,-2.8) circle (0pt) node{\tiny{$p_{15}$}};
\filldraw[black] (11.5,-2.8) circle (0pt) node{\tiny{$p_{16}$}};
\filldraw[black] (14.75,-2.8) circle (0pt) node{\tiny{$p_{17}$}};
\filldraw[black] (18.25,-2.8) circle (0pt) node{\tiny{$p_{18}$}};
\filldraw[black] (21.5,-2.8) circle (0pt) node{\tiny{$p_{19}$}};
\filldraw[black] (24,-2.8) circle (0pt) node{\tiny{$p_{20}$}};
\filldraw[black] (26.5,-2.8) circle (0pt) node{\tiny{$p_{21}$}};
\filldraw[black] (29,-2.8) circle (0pt) node{\tiny{$p_{22}$}};
\filldraw[black] (31.5,-2.8) circle (0pt) node{\tiny{$p_{23}$}};
\filldraw[black] (34,-2.8) circle (0pt) node{\tiny{$p_{24}$}};

\end{tikzpicture}
\caption{A collection of 24 rectangle tiles.} 
\label{F_fibonacci_substitution_prototiles}
\end{figure}

Figure \ref{F_fibonacci_substitution_supertiles} demonstrates a substitution over $\{p_1,\dots,p_{24}\}$. This substitution and its domain are denoted by $\omega$ and $\cP_{\omega}$, respectively, throughout the document. For simplicity, we will represent tiles in $\cP_{\omega}$ by their labels only, because labels are sufficient enough to indicate which tiles we are referring to. For example, $\omega(p_1) = \omega(A^+_1)$ is a collection of four rectangle tiles that are (translated copies of) $A^-_4, B^+_1, D^+_2$ and $C^-_1$.

Labels are essential to see the relation between $\omega$ and $\mu_2$. Observe that $\psi \circ \omega = \mu_2$ where $\psi$ is a function defined on $\cP_{\omega}$ such that
\begin{align*}
\psi(A^j_i) = A & \text{\ for\ all\ }j\in\{-,+\}\ \text{and}\ i\in\{1,2,3,4\}, \\
\psi(B^j_i) = B & \text{\ for\ all\ }j\in\{-,+\}\ \text{and}\ i\in\{1,2\}, \\
\psi(C^j_i) = C & \text{\ for\ all\ }j\in\{-,+\}\ \text{and}\ i\in\{1,2\}, \\
\psi(D^j_i) = D & \text{\ for\ all\ }j\in\{-,+\}\ \text{and}\ i\in\{1,2,3,4\}.
\end{align*}
In fact, $\omega$ is a substitution defined over decorated tiles of $\mu_2$.

\begin{figure}[H]
\centering
\begin{tikzpicture}[scale=0.4]
\centering
\SFAI{-0.5}{0}{0}{1};
\SFAII{3}{0}{0}{1};
\SFAIII{7}{0}{0}{1};
\SFAIV{10.5}{0}{0}{1};

\SFBI{14}{0}{0}{1};
\SFBII{17.5}{0}{0}{1};
\SFCI{21}{0}{0}{1};
\SFCII{23.5}{0}{0}{1};

\SFDI{26}{0}{0}{1};
\SFDII{28.5}{0}{0}{1};
\SFDIII{31}{0}{0}{1};
\SFDIV{33.5}{0}{0}{1};

\filldraw[black] (0.75,3.5) circle (0pt) node{\tiny{$\omega(A^+_1)$}};
\filldraw[black] (4.5,3.5) circle (0pt) node{\tiny{$\omega(A^+_2)$}};
\filldraw[black] (8.3,3.5) circle (0pt) node{\tiny{$\omega(A^+_3)$}};
\filldraw[black] (12,3.5) circle (0pt) node{\tiny{$\omega(A^+_4)$}};
\filldraw[black] (15.25,3.5) circle (0pt) node{\tiny{$\omega(B^+_1)$}};
\filldraw[black] (18.75,3.5) circle (0pt) node{\tiny{$\omega(B^+_2)$}};
\filldraw[black] (22,3.5) circle (0pt) node{\tiny{$\omega(C^+_1)$}};
\filldraw[black] (24.5,3.5) circle (0pt) node{\tiny{$\omega(C^+_2)$}};
\filldraw[black] (27,3.5) circle (0pt) node{\tiny{$\omega(D^+_1)$}};
\filldraw[black] (29.5,3.5) circle (0pt) node{\tiny{$\omega(D^+_2)$}};
\filldraw[black] (32,3.5) circle (0pt) node{\tiny{$\omega(D^+_3)$}};
\filldraw[black] (34.5,3.5) circle (0pt) node{\tiny{$\omega(D^+_4)$}};

\SFARI{-0.5}{-5.5}{0}{1};
\SFARII{3}{-5.5}{0}{1};
\SFARIII{7}{-5.5}{0}{1};
\SFARIV{10.5}{-5.5}{0}{1};

\SFBRI{14}{-5.5}{0}{1};
\SFBRII{17.5}{-5.5}{0}{1};
\SFCRI{21}{-5.5}{0}{1};
\SFCRII{23.5}{-5.5}{0}{1};

\SFDRI{26}{-5.5}{0}{1};
\SFDRII{28.5}{-5.5}{0}{1};
\SFDRIII{31}{-5.5}{0}{1};
\SFDRIV{33.5}{-5.5}{0}{1};

\filldraw[black] (0.75,-2) circle (0pt) node{\tiny{$\omega(A^-_1)$}};
\filldraw[black] (4.5,-2) circle (0pt) node{\tiny{$\omega(A^-_2)$}};
\filldraw[black] (8.3,-2) circle (0pt) node{\tiny{$\omega(A^-_3)$}};
\filldraw[black] (12,-2) circle (0pt) node{\tiny{$\omega(A^-_4)$}};
\filldraw[black] (15.25,-2) circle (0pt) node{\tiny{$\omega(B^-_1)$}};
\filldraw[black] (18.75,-2) circle (0pt) node{\tiny{$\omega(B^-_2)$}};
\filldraw[black] (22,-2) circle (0pt) node{\tiny{$\omega(C^-_1)$}};
\filldraw[black] (24.5,-2) circle (0pt) node{\tiny{$\omega(C^-_2)$}};
\filldraw[black] (27,-2) circle (0pt) node{\tiny{$\omega(D^-_1)$}};
\filldraw[black] (29.5,-2) circle (0pt) node{\tiny{$\omega(D^-_2)$}};
\filldraw[black] (32,-2) circle (0pt) node{\tiny{$\omega(D^-_3)$}};
\filldraw[black] (34.5,-2) circle (0pt) node{\tiny{$\omega(D^-_4)$}};

\end{tikzpicture}
\caption{An iteration rule $\omega$ defined for the 24 rectangle tiles introduced in Figure \ref{F_fibonacci_substitution_prototiles} respectively.}
\label{F_fibonacci_substitution_supertiles}
\end{figure}
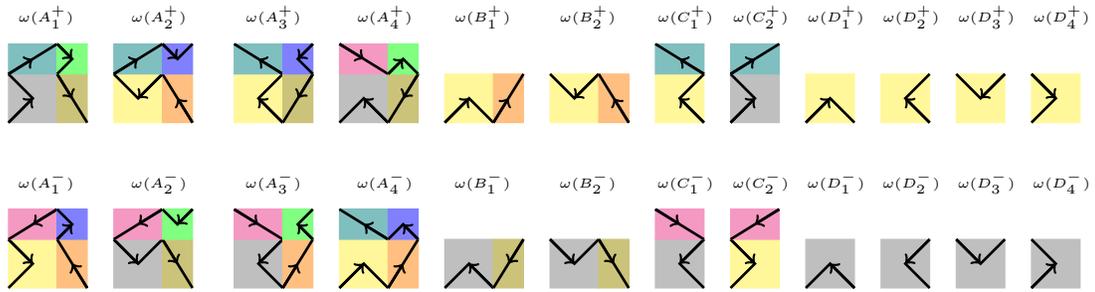

Note that attached curves on 1-supertiles in Figure \ref{F_fibonacci_substitution_supertiles} are designed so that they can be concatenated to form single curves. This is true in general for each $k$-supertile for $k\in\ZZ^+$. In particular, we have the following result.

\begin{proposition}\label{proposition_supertile_decoration}
Let $k\in\ZZ^+$ and $p\in\cP_{\omega}$ be fixed. Suppose $p$ has a curve (decoration) $e_p$ attached on it with a start point $x\in\RR^2$ and an end point $y\in\RR^2$. Then $\omega^k(p)$ consists of decorated tiles whose attached curves (decorations) can be concatenated to form a single curve with a start point $\phi^kx$ and an end point $\phi^ky$, where $\phi=(1+\sqrt{5})/2$.
\end{proposition}

\begin{proof}
The statement holds for $k=1$ as it can be readily observed from Figure \ref{F_fibonacci_substitution_supertiles}. Assume that the statement holds for each $q\in\cP_{\omega}$ and $i\in\ZZ^+$ with $i<m$ for some $m>1$. Let $p\in\cP_{\omega}$ be fixed. Suppose $t_1,t_2\in\omega(p)$ such that $e_1, e_2$ are decorations of $t_1,t_2$, respectively so that $r(e_1)=s(e_2)=a\in\RR^2$ where $r(e_1)$ is the end point of $e_1$ and $s(e_2)$ is the starting point of $e_2$. By induction assumption, $\omega^{m-1}(t_1)$ consists of tiles whose attached curves (decorations) can be concatenated to form a single curve $C_1$ with an end point $\phi^{m-1}a$, and $\omega^{m-1}(t_2)$ consists of tiles whose attached curves (decorations) can be concatenated to form a single curve $C_2$ with a start point $\phi^{m-1}a$. That is, $C_1$ and $C_2$ can be concatenated. Hence, the statement holds for $m>1$ as well because $\omega^m(p)=\bigcup\limits_{t\in\omega(p)}\omega^{m-1}(t)$.
\end{proof}

Under the lights of Proposition \ref{proposition_supertile_decoration}, we are able to define decorations for supertiles of $\omega$.

\begin{definition}
Let $k\in\ZZ^+$ and $p\in\cP_{\omega}$ be fixed. The single curve formed by concatenation of decorations of tiles in $\omega^{k}(p)$ is called the \emph{decoration} of $\omega^k(p)$.
\end{definition}

\begin{remark}
For each $k\in\ZZ^+$, a decoration $C_k$ of a $k$-supertile $\omega^k(p)$ for $p\in\cP_{\omega}$ visits every tile in $\omega^k(p)$ exactly once. As such, $C_k$ induces a natural total order $\lesssim$ between tiles of $\omega^k(p)$ such that
\[
t_1\lesssim t_2\quad\text{if\ }t_1\text{\ is\ visited\ by\ }C_k\text{\ before\ }t_2
\]
for each distinct pair $t_1, t_2\in\omega^k(p)$. For example, decoration of $\omega(A_1^+)$ induces a total order between $\{A^-_4, B^+_1 , D^+_4, C^-_1\}$ such that
\[
A^-_4\lesssim B^+_1 \lesssim D^+_4\lesssim C^-_1.
\]
\end{remark}

\paragraph{A Tessellation of the Plane}
Start with the tile $p_1=A^+_1$. Iterating $A^+_1$ twice generates a 3x3 grid of tiles (i.e. $\omega^2(A^+_1)$) as shown on the middle of Figure \ref{figure_tesselation_quater}. Note that $A^+_1$ is invariant in this process. By the same token, $\omega^2(A^+_1)$ is contained in $\omega^4(A^+_1)$ as shown on the right side of Figure \ref{figure_tesselation_quater}. Iterating ad inifinitum covers the first quadrant of the plane $\{(a,b):\ a,b\geq0\}$. Its reflections over the lines $x=0$, $y=0$ and $x=-y$ cover the other quadrants, and generates a tessellation of the plane.

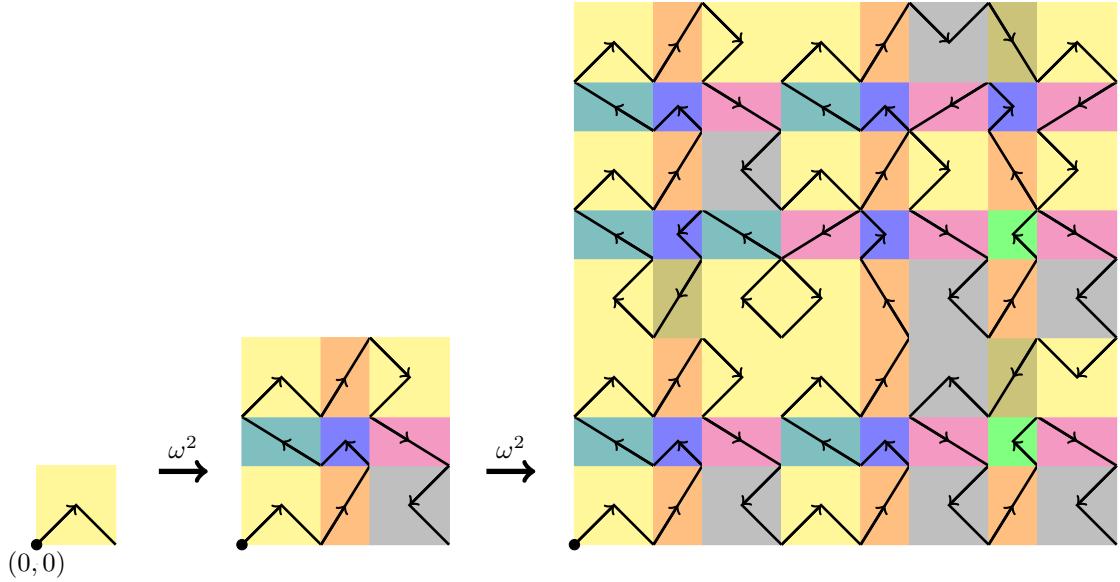
\begin{figure}[htp]
\centering
\begin{tikzpicture}[scale=0.65]
\centering

\FAI{-1}{0}{0}{1};

\SSFAI{3.2}{0}{0}{1};

\SSSSFAI{10}{0}{0}{1};
\draw[line width=2pt, color=black, ->] (1.5,1.5) -- (2.5,1.5);
\draw[line width=2pt, color=black, ->] (8.2,1.5) -- (9.2,1.5);

\filldraw[black] (2,2) circle (0pt) node{$\omega^2$};
\filldraw[black] (8.7,2) circle (0pt) node{$\omega^2$};

\filldraw[black] (-1,-0.4) circle (0pt) node{$(0,0)$};
\filldraw[black] (-1,0) circle (3pt) ;
\filldraw[black] (3.2,0) circle (3pt) ;
\filldraw[black] (10,0) circle (3pt) ;
\end{tikzpicture}
\caption{A tessellation of the first quadrant of the plane}
\label{figure_tesselation_quater}
\end{figure}

Note that a neater version of decorations of supertiles can be generated by marking each tile with its centre and connecting marked points through lines successively via the total order induced by decorations of supertiles (Figure \ref{figure_tesselation_quater_neat}). Note also that there are various other possible tessellations of the plane that can be constructed by $\omega$ \cite[Theorem 1.4]{sadun2008topology}.

\paragraph{One-dimensional Representation of $\omega$} 
The map $\omega$ induces a one dimensional substitution $\nu$ whenever it is read via the total order induced from decorations of 1-supertiles. Precisely,

\begin{align*}
    A_1^+ & \mapsto A_4^-, B_1^+, D_2^+, C_1^-, & A_1^- &\mapsto C_1^+, D_2^-, B_1^-, A_4^+, & D_1^+ &\mapsto A_1^+, \\
    A_2^+ & \mapsto C_1^+, A_3^+, B_1^+, D_1^-, & A_2^- &\mapsto D_1^+, B_1^-, A_3^-, C_1^-, & D_2^+ &\mapsto A_2^+, \\
    A_3^+ & \mapsto D_2^-, C_2^-, A_2^+, B_2^+, & A_3^- &\mapsto B_2^-, A_2^-, C_2^+, D_2^+, & D_3^+ &\mapsto A_3^+, \\
    A_4^+ & \mapsto B_2^-, D_1^+, C_2^-, A_1^-, & A_4^- &\mapsto A_1^+, C_2^+, D_1^-, B_2^+, & D_4^+ &\mapsto A_4^+, \\
    B_1^+ & \mapsto A_1^+, C_2^+, & B_1^- &\mapsto C_2^-, A_1^-, & D_1^- &\mapsto A_1^-, \\
    B_2^+ & \mapsto C_1^+, A_3^+, & B_2^- &\mapsto A_3^-, C_1^-, & D_2^- &\mapsto A_2^-, \\
    C_1^+ & \mapsto A_2^+, B_2^+, & C_1^- &\mapsto B_2^-, A_2^-, & D_3^- &\mapsto A_3^-, \\
    C_2^+ & \mapsto A_4^-, B_1^+, & C_2^- &\mapsto B_1^-, A_4^+, & D_4^- &\mapsto A_4^-.
\end{align*}
The map $\nu$ (or $\omega$) induces a $24\times24$ matrix (Figure \ref{figure_matrix_rep}), whose $(i,j)$-th entry is the number of times $p_j$ appears in $\nu(p_i)$ (or $\omega(p_i)$) with the notation convention given in Figure \ref{F_fibonacci_substitution_prototiles}. It represents the map $\nu$ (or $\omega$) and its Perron-Frobenius eigenvalue is $\phi^2$ with multiplicity 1, where $\phi=(1+\sqrt{5})/2$, as a consequence of Perron-Frobenius Theorem \cite[Theorem 2 and Theorem 3]{frettloh2017substitution}.
\begin{figure}[H]
\centering
\[
\begin{bmatrix}
0 & 0 & 0 & 0 & 1 & 0 & 0 & 0 & 0 & 1 & 0 & 0 & 0 & 0 & 0 & 1 & 0 & 0 & 1 & 0 & 0 & 0 & 0 & 0\\
0 & 0 & 1 & 0 & 1 & 0 & 1 & 0 & 0 & 0 & 0 & 0 & 0 & 0 & 0 & 0 & 0 & 0 & 0 & 0 & 1 & 0 & 0 & 0\\
0 & 1 & 0 & 0 & 0 & 1 & 0 & 0 & 0 & 0 & 0 & 0 & 0 & 0 & 0 & 0 & 0 & 0 & 0 & 1 & 0 & 1 & 0 & 0\\
0 & 0 & 0 & 0 & 0 & 0 & 0 & 0 & 1 & 0 & 0 & 0 & 1 & 0 & 0 & 0 & 0 & 1 & 0 & 1 & 0 & 0 & 0 & 0\\
1 & 0 & 0 & 0 & 0 & 0 & 0 & 1 & 0 & 0 & 0 & 0 & 0 & 0 & 0 & 0 & 0 & 0 & 0 & 0 & 0 & 0 & 0 & 0\\
0 & 0 & 1 & 0 & 0 & 0 & 1 & 0 & 0 & 0 & 0 & 0 & 0 & 0 & 0 & 0 & 0 & 0 & 0 & 0 & 0 & 0 & 0 & 0\\
0 & 1 & 0 & 0 & 0 & 1 & 0 & 0 & 0 & 0 & 0 & 0 & 0 & 0 & 0 & 0 & 0 & 0 & 0 & 0 & 0 & 0 & 0 & 0\\
0 & 0 & 0 & 0 & 1 & 0 & 0 & 0 & 0 & 0 & 0 & 0 & 0 & 0 & 0 & 1 & 0 & 0 & 0 & 0 & 0 & 0 & 0 & 0\\
1 & 0 & 0 & 0 & 0 & 0 & 0 & 0 & 0 & 0 & 0 & 0 & 0 & 0 & 0 & 0 & 0 & 0 & 0 & 0 & 0 & 0 & 0 & 0\\
0 & 1 & 0 & 0 & 0 & 0 & 0 & 0 & 0 & 0 & 0 & 0 & 0 & 0 & 0 & 0 & 0 & 0 & 0 & 0 & 0 & 0 & 0 & 0\\
0 & 0 & 1 & 0 & 0 & 0 & 0 & 0 & 0 & 0 & 0 & 0 & 0 & 0 & 0 & 0 & 0 & 0 & 0 & 0 & 0 & 0 & 0 & 0\\
0 & 0 & 0 & 1 & 0 & 0 & 0 & 0 & 0 & 0 & 0 & 0 & 0 & 0 & 0 & 0 & 0 & 0 & 0 & 0 & 0 & 0 & 0 & 0\\
0 & 0 & 0 & 1 & 0 & 0 & 1 & 0 & 0 & 0 & 0 & 0 & 0 & 0 & 0 & 0 & 1 & 0 & 0 & 0 & 0 & 1 & 0 & 0\\
0 & 0 & 0 & 0 & 0 & 0 & 0 & 0 & 1 & 0 & 0 & 0 & 0 & 0 & 1 & 0 & 1 & 0 & 1 & 0 & 0 & 0 & 0 & 0\\
0 & 0 & 0 & 0 & 0 & 0 & 0 & 1 & 0 & 1 & 0 & 0 & 0 & 1 & 0 & 0 & 0 & 1 & 0 & 0 & 0 & 0 & 0 & 0\\
1 & 0 & 0 & 0 & 0 & 1 & 0 & 1 & 0 & 0 & 0 & 0 & 0 & 0 & 0 & 0 & 0 & 0 & 0 & 0 & 1 & 0 & 0 & 0\\
0 & 0 & 0 & 0 & 0 & 0 & 0 & 0 & 0 & 0 & 0 & 0 & 1 & 0 & 0 & 0 & 0 & 0 & 0 & 1 & 0 & 0 & 0 & 0\\
0 & 0 & 0 & 0 & 0 & 0 & 0 & 0 & 0 & 0 & 0 & 0 & 0 & 0 & 1 & 0 & 0 & 0 & 1 & 0 & 0 & 0 & 0 & 0\\
0 & 0 & 0 & 0 & 0 & 0 & 0 & 0 & 0 & 0 & 0 & 0 & 0 & 1 & 0 & 0 & 0 & 1 & 0 & 0 & 0 & 0 & 0 & 0\\
0 & 0 & 0 & 1 & 0 & 0 & 0 & 0 & 0 & 0 & 0 & 0 & 0 & 0 & 0 & 0 & 1 & 0 & 0 & 0 & 0 & 0 & 0 & 0\\
0 & 0 & 0 & 0 & 0 & 0 & 0 & 0 & 0 & 0 & 0 & 0 & 1 & 0 & 0 & 0 & 0 & 0 & 0 & 0 & 0 & 0 & 0 & 0\\
0 & 0 & 0 & 0 & 0 & 0 & 0 & 0 & 0 & 0 & 0 & 0 & 0 & 1 & 0 & 0 & 0 & 0 & 0 & 0 & 0 & 0 & 0 & 0\\
0 & 0 & 0 & 0 & 0 & 0 & 0 & 0 & 0 & 0 & 0 & 0 & 0 & 0 & 1 & 0 & 0 & 0 & 0 & 0 & 0 & 0 & 0 & 0\\
0 & 0 & 0 & 0 & 0 & 0 & 0 & 0 & 0 & 0 & 0 & 0 & 0 & 0 & 0 & 1 & 0 & 0 & 0 & 0 & 0 & 0 & 0 & 0
\end{bmatrix}
\]
\caption{Matrix representation of $\nu$ and $\omega$.}
\label{figure_matrix_rep}
\end{figure}
\begin{figure}[H]
\centering
\begin{tikzpicture}[scale=0.65]
\centering

\TAI{-1}{0}{0}{1};

\SSTAI{3.2}{0}{0}{1};

\SSSSTAI{10}{0}{0}{1};
\draw[line width=2pt, color=black, ->] (1.5,1.5) -- (2.5,1.5);
\draw[line width=2pt, color=black, ->] (8.2,1.5) -- (9.2,1.5);

\filldraw[black] (2,2) circle (0pt) node{$\omega^2$};
\filldraw[black] (8.7,2) circle (0pt) node{$\omega^2$};

\filldraw[black] (-1,-0.4) circle (0pt) node{$(0,0)$};
\filldraw[black] (-1,0) circle (3pt) ;
\filldraw[black] (3.2,0) circle (3pt) ;
\filldraw[black] (10,0) circle (3pt) ;

\draw[line width=2pt, color=black, ->] (3.2+1.61803398875/2,1.61803398875/2) -- (3.2+1.61803398875+0.5,1.61803398875/2) -- (3.2+1.61803398875+0.5,1.61803398875+0.5) -- (3.2+1.61803398875/2,1.61803398875+0.5) -- (3.2+1.61803398875/2,1.61803398875/2+1.61803398875+1) -- (3.2+1.61803398875/2+1.61803398875+1,1.61803398875/2+1.61803398875+1) -- (3.2+1.61803398875/2+1.61803398875+1,1.61803398875/2);
\draw[line width=2pt, color=black, ->] (10+1.61803398875/2,1.61803398875/2) -- (10+1.61803398875+0.5,1.61803398875/2) -- (10+1.61803398875+0.5,1.61803398875+0.5) -- (10+1.61803398875/2,1.61803398875+0.5) -- (10+1.61803398875/2,1.61803398875/2+1.61803398875+1) -- (10+1.61803398875/2+1.61803398875+1,1.61803398875/2+1.61803398875+1) -- (10+1.61803398875/2+1.61803398875+1,1.61803398875/2) -- (10+3*1.61803398875+1.5,1.61803398875/2) -- (10+3*1.61803398875+1.5,1.61803398875+0.5) -- (10+2.5*1.61803398875+1,1.61803398875+0.5)-- (10+2.5*1.61803398875+1,1.5*1.61803398875+1) -- (10+3*1.61803398875+1.5,1.5*1.61803398875+1) -- (10+3*1.61803398875+1.5,3*1.61803398875+1.5) -- (10+2.5*1.61803398875+1,3*1.61803398875+1.5) -- (10+2.5*1.61803398875+1,2.5*1.61803398875+1) -- (10+1.5*1.61803398875+1,2.5*1.61803398875+1) -- (10+1.5*1.61803398875+1,3*1.61803398875+1.5) -- (10+1.61803398875+0.5,3*1.61803398875+1.5) -- (10+1.61803398875+0.5,2.5*1.61803398875+1) -- (10+0.5*1.61803398875,2.5*1.61803398875+1) -- (10+0.5*1.61803398875,3.5*1.61803398875+2) -- (10+1.61803398875+0.5,3.5*1.61803398875+2) -- (10+1.61803398875+0.5,4*1.61803398875+2.5) -- (10+0.5*1.61803398875,4*1.61803398875+2.5) -- (10+0.5*1.61803398875,4.5*1.61803398875+3) -- (10+0.5*1.61803398875,4.5*1.61803398875+3) -- (10+1.5*1.61803398875+1,4.5*1.61803398875+3) -- (10+1.5*1.61803398875+1,3.5*1.61803398875+2) -- (10+3*1.61803398875+1.5,3.5*1.61803398875+2) -- (10+3*1.61803398875+1.5,4*1.61803398875+2.5) -- (10+2.5*1.61803398875+1,4*1.61803398875+2.5) -- (10+2.5*1.61803398875+1,4.5*1.61803398875+3) -- (10+4.5*1.61803398875+3,4.5*1.61803398875+3) -- (10+4.5*1.61803398875+3,3.5*1.61803398875+2) -- (10+4*1.61803398875+2.5,3.5*1.61803398875+2) -- (10+4*1.61803398875+2.5,4*1.61803398875+2.5) -- (10+3.5*1.61803398875+2,4*1.61803398875+2.5) -- (10+3.5*1.61803398875+2,2.5*1.61803398875+1) -- (10+4*1.61803398875+2.5,2.5*1.61803398875+1) -- (10+4*1.61803398875+2.5,3*1.61803398875+1.5) -- (10+4.5*1.61803398875+3,3*1.61803398875+1.5) -- (10+4.5*1.61803398875+3,1.5*1.61803398875+1) -- (10+3.5*1.61803398875+2,1.5*1.61803398875+1) -- (10+3.5*1.61803398875+2,0.5*1.61803398875) -- (10+4*1.61803398875+2.5,0.5*1.61803398875) -- (10+4*1.61803398875+2.5,1.61803398875+0.5) -- (10+4.5*1.61803398875+3,1.61803398875+0.5) -- (10+4.5*1.61803398875+3,0.5*1.61803398875);
\end{tikzpicture}
\caption{A neat visual of the tessellation of the first quadrant of the plane.} 
\label{figure_tesselation_quater_neat}
\end{figure}

\section{The Fibonacci Space-Filling Curve}

In this section we generate a space-filling curve using $\omega$.

\begin{theorem}\label{t_fibonacci_sfc} The substitution $\omega$ induces a continuous surjection (i.e. a space-filling curve) $F:[0,1]\mapsto[0,1]\times[0,1]$. 
\end{theorem}

\begin{proof}
Consider the collection of scaled patches $\{\phi^{-k-1}\omega^k(A^+_1):\ k\in\NN\}$ where $\phi=(1+\sqrt{5})/2$. For each $k$, the scaled patch $\phi^{-k-1}\omega^k(A^+_1)$ defines a partition of $[0,1]\times[0,1]$ and contains $\cF^2_{k+2}$ many tiles where $\cF_i$ is the $i$-th Fibonacci number. Additionally, there is a (scaled) decoration over the patch which visits every tile exactly once and induces a total order $\lesssim_k$ amongst them. Let $J^k_1,\dots,J^k_{\cF_{k+2}}$ denote the (scaled) tiles in $\phi^{-k-1}\omega^k(A^+_1)$ such that $J^k_i\lesssim_k J^k_j$ whenever $i<j$ (Figure \ref{Figure_Labelling_the_Rectangles_in_the_partititons_of_the_unit_square}).
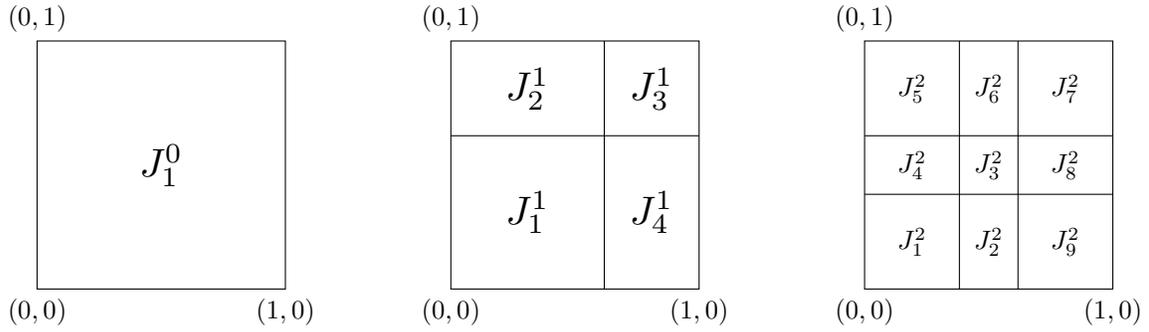
\begin{figure}[H]
\centering
\begin{tikzpicture}[scale=0.55]

\draw(0-10,0) rectangle (6-10,6);
\node[scale=1.5] at (3-10,3) {$J_{1}^0$};
\node[anchor=north] at (0-10,0) {$(0,0)$};
\node[anchor=north] at (6-10,0) {$(1,0)$};
\node[anchor=south] at (0-10,6) {$(0,1)$};

\foreach \index in {0,10}{
\draw (0+\index,0) rectangle (6+\index,6);
\draw (\index+6/1.6180339887498948482045868343656381177203091798057628621354486227,0) -- (\index+6/1.6180339887498948482045868343656381177203091798057628621354486227,6);
\draw (\index,6/1.6180339887498948482045868343656381177203091798057628621354486227) -- (\index+6,6/1.6180339887498948482045868343656381177203091798057628621354486227);
\node[anchor=north] at (0+\index,0) {$(0,0)$};
\node[anchor=north] at (6+\index,0) {$(1,0)$};
\node[anchor=south] at (0+\index,6) {$(0,1)$};

}
\draw (10+6/2.6180339887498948482045868343656381177203091798057628621354486227,0) -- (10+6/2.6180339887498948482045868343656381177203091798057628621354486227,6);
\draw (10,6/2.6180339887498948482045868343656381177203091798057628621354486227) -- (10+6,6/2.6180339887498948482045868343656381177203091798057628621354486227);

\node[scale=1.5] at (3/1.6180339887498948482045868343656381177203091798057628621354486227,3/1.6180339887498948482045868343656381177203091798057628621354486227) {$J_{1}^1$};
\node[scale=1.5] at (3/1.6180339887498948482045868343656381177203091798057628621354486227,3*1.6180339887498948482045868343656381177203091798057628621354486227) {$J_{2}^1$};
\node[scale=1.5] at (3*1.6180339887498948482045868343656381177203091798057628621354486227,3/1.6180339887498948482045868343656381177203091798057628621354486227) {$J_{4}^1$};
\node[scale=1.5] at (3*1.6180339887498948482045868343656381177203091798057628621354486227,3*1.6180339887498948482045868343656381177203091798057628621354486227) {$J_{3}^1$};

\node[scale=1] at (10+3/2.6180339887498948482045868343656381177203091798057628621354486227,3/2.6180339887498948482045868343656381177203091798057628621354486227) {$J_{1}^2$};
\node[scale=1] at (10+3/2.6180339887498948482045868343656381177203091798057628621354486227,3) {$J_{4}^2$};
\node[scale=1] at (10+3,3/2.6180339887498948482045868343656381177203091798057628621354486227) {$J_{2}^2$};
\node[scale=1] at (10+3,3) {$J_{3}^2$};
\node[scale=1] at (10+3/2.6180339887498948482045868343656381177203091798057628621354486227,3*1.6180339887498948482045868343656381177203091798057628621354486227) {$J_{5}^2$};
\node[scale=1] at (10+3,3*1.6180339887498948482045868343656381177203091798057628621354486227) {$J_{6}^2$};
\node[scale=1] at (10+3*1.6180339887498948482045868343656381177203091798057628621354486227,3/2.6180339887498948482045868343656381177203091798057628621354486227) {$J_{9}^2$};
\node[scale=1] at (10+3*1.6180339887498948482045868343656381177203091798057628621354486227,3) {$J_{8}^2$};
\node[scale=1] at (10+3*1.6180339887498948482045868343656381177203091798057628621354486227,3*1.6180339887498948482045868343656381177203091798057628621354486227) {$J_{7}^2$};
\end{tikzpicture}
\caption{Partitions $\phi^{-k-1}\omega^k(A^+_1)$ for $k=0,1,2$ are demonstrated.}
\label{Figure_Labelling_the_Rectangles_in_the_partititons_of_the_unit_square}
\end{figure}
Next, we define a sequence of partitions of $[0,1]$. For each $k\in\NN$, define a partition $\{I^k_1,\dots,I^k_{\cF^2_{k+2}}\}$ such that
\begin{itemize}
    \item[(i)] Length of $I^k_i$ equals to area of $J^k_i$,
    \item[(ii)] $I^k_i$ is a closed interval for each $i\in\{1,\dots,\cF^2_{k+2}\}$,
    \item[(iii)] $\min (I^k_i) < \min (I^k_j)$ if and only if $i<j$, for each $i,j\in\{1,\dots,\cF^2_{k+2}\}$.
\end{itemize}
The partitions $\{I^0_1\}$, $\{I^1_1, I^1_2, I^1_3, I^1_4\}$ and $\{I^2_1, I^2_2, I^2_3, I^2_4, I^2_5, I^2_6, I^2_7, I^2_8, I^2_9\}$ are illustrated in Figure \ref{Figure_Labelling_the_Rectangles_in_the_partititons_of_the_unit_interval}.

\begin{figure}[H]
\centering
\begin{tikzpicture}[scale=0.82]
\foreach \inx in {0,-3,-6}{
\draw (0,0+\inx) -- (8*1.618 + 5,0+\inx);
\draw(0,0.2+\inx) -- (0,-0.2+\inx);
\draw(8*1.618 + 5,0.2+\inx) -- (8*1.618 + 5,-0.2+\inx);
\node[anchor=north] at (0,-0.2+\inx) {$0$};
\node[anchor=north] at (8*1.618 + 5,-0.2+\inx) {$1$};
}

\foreach \inx in {-3,-6}{
\draw(3*1.618+2,0.2+\inx) -- (3*1.618+2,-0.2+\inx);
\draw(5*1.618+3,0.2+\inx) -- (5*1.618+3,-0.2+\inx);
\draw(6*1.618+4,0.2+\inx) -- (6*1.618+4,-0.2+\inx);
}

\draw(1*1.618+1,0.2-6) -- (1*1.618+1,-0.2-6);
\draw(2*1.618+1,0.2-6) -- (2*1.618+1,-0.2-6);
\draw(2*1.618+2,0.2-6) -- (2*1.618+2,-0.2-6);
\draw(4*1.618+3,0.2-6) -- (4*1.618+3,-0.2-6);
\draw(7*1.618+4,0.2-6) -- (7*1.618+4,-0.2-6);

\node[scale=1.5, anchor=south] at (4*1.618 + 2.5,0) {$I^0_1$};

\node[scale=1.5, anchor=south] at (1.5*1.618+1,-3) {$I^1_1$};
\node[scale=1.5, anchor=south] at (4*1.618+2.5,-3) {$I^1_2$};
\node[scale=1.5, anchor=south] at (5.5*1.618+3.5,-3) {$I^1_3$};
\node[scale=1.5, anchor=south] at (7*1.618+4.5,-3) {$I^1_4$};

\node[anchor=south] at (0.5*1.618+0.5,-6) {$I^2_1$};
\node[anchor=south] at (1.5*1.618+1,-6) {$I^2_2$};
\node[anchor=south] at (2*1.618+1.5,-6) {$I^2_3$};
\node[anchor=south] at (2.5*1.618+2,-6) {$I^2_4$};
\node[anchor=south] at (3.5*1.618+2.5,-6) {$I^2_5$};
\node[anchor=south] at (4.5*1.618+3,-6) {$I^2_6$};
\node[anchor=south] at (5.5*1.618+3.5,-6) {$I^2_7$};
\node[anchor=south] at (6.5*1.618+4,-6) {$I^2_8$};
\node[anchor=south] at (7.5*1.618+4.5,-6) {$I^2_9$};

\end{tikzpicture}

\caption{Partitions $\{I^0_1\}$, $\{I^1_1, I^1_2, I^1_3, I^1_4\}$ and $\{I^2_1, I^2_2, I^2_3, I^2_4, I^2_5, I^2_6, I^2_7, I^2_8, I^2_9\}$ are depicted.}
\label{Figure_Labelling_the_Rectangles_in_the_partititons_of_the_unit_interval}
\end{figure}
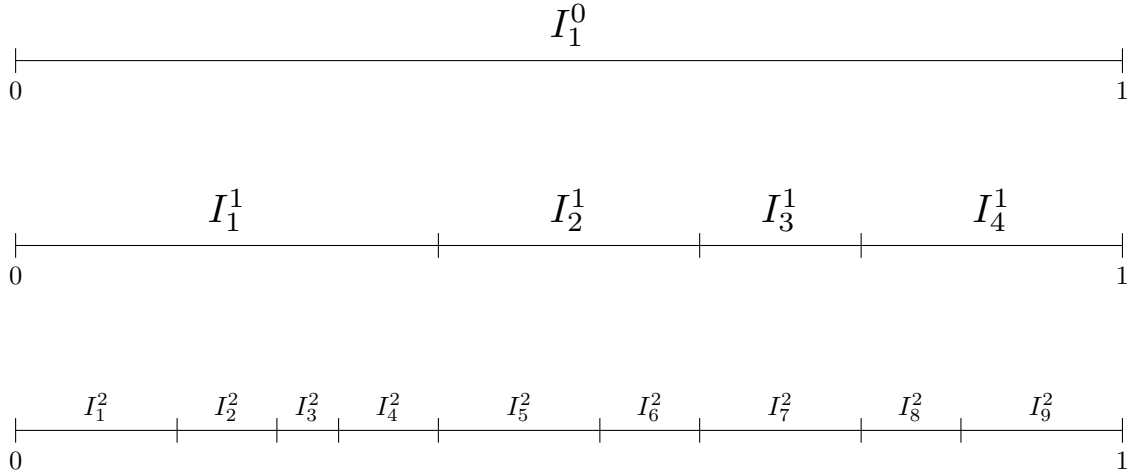

For each $k\in\NN$, define a map $F_k:\{I^k_i:\ i=1,\dots,\cF_{k+2}^2\}\mapsto \{J^k_i:\ i=1,\dots,\cF_{k+2}^2\}$ such that $F_k(I^k_i)=J^k_i$ for $i\in\{1,\dots,\cF_{k+2}^2\}$. By Cantor's intersection theorem, for each $x\in[0,1]$, there exists a sequence of nested intervals $\{I^k_{m_k}:\ k\in\ZZ^+\ \text{and}\ m_k\in\{1,\dots,\cF^2_{k+2}\}\}$ such that 
\begin{itemize}
    \item[(i)] $I^0_1\supseteq I^1_{m_1}\supseteq I^2_{m_2}\supseteq\dots$, 
    \item[(ii)] $\bigcap\limits_{k=1}I^k_{m_k}=\{x\}$. 
\end{itemize}
Similarly, for each $y\in[0,1]\times[0,1]$, there exists a sequence of nested rectangles $\{J^k_{m_k}:\ k\in\ZZ^+\ \text{and}\ m_k\in\{1,\dots,\cF^2_{k+2}\}\}$ such that 
\begin{itemize}
    \item[(i)] $J^0_1\supseteq J^1_{m_1}\supseteq j^2_{m_2}\supseteq\dots$, 
    \item[(ii)] $\bigcap\limits_{k=1}J^k_{m_k}=\{y\}$. 
\end{itemize}
Consequently, we define a function $F:[0,1]\mapsto[0,1]\times[0,1]$ by the following relation
\[
F(x)=\bigcap\limits_{k=1}J^k_{m_k}\quad\text{for}\ x\in[0,1]\ \text{with}\ \{x\}=\bigcap\limits_{k=1}I^k_{m_k}.
\]
The function $F$ is well-defined. It is surjective because $\bigcap\limits_{k=1}I^k_{m_k}$ is a singleton for each $y\in[0,1]\times[0,1]$ with $\{y\}=\bigcap\limits_{k=1}J^k_{m_k}$, by Cantor's intersection theorem. That is, $F(\bigcap\limits_{k=1}I^k_{m_k})=y$. 

Next, we show that it is also continuous. For each $k\in\NN$, let $g_k$ denote the maximum length of the intervals $I^k_1,\dots,I^k_{\cF^2_{k+2}}$, and let $h_k$ denote the maximum diameter of the rectangles $J^k_1,\dots,J^k_{\cF^2_{k+2}}$. Note that both $g_k\downarrow0$ and $h_k\downarrow0$ as $k\to\infty$. Suppose $x\in[0,1]$ and $\epsilon>0$ are given. Choose a sufficiently large $N\in\NN$ such that $h_N<\epsilon$. Then $||F(x)-F(y)||<\epsilon$ whenever $|x-y|<g_N/2$, and $F$ is continuous.
\end{proof}

\paragraph{Connected Space-Filling Curves} A space-filling $f$ in $\RR^2$ is called \emph{connected} if any two subsequent rectangles that are generated on iterative partitions (i.e. images of subsequent adjacent intervals) share a common edge \cite[Definition 7.2]{bader2012space}. One of the most famous example of a connected space-filling curve is the Hilbert's space-filling curve. For example, the curves illustrated in Figure \ref{f_Hilbert_SFC_iterations} visit adjacent rectangles that share an edge respectively. Similarly, the space-filling curve $F$ constructed in the proof of Theorem \ref{t_fibonacci_sfc} is also a connected space-filling curve (Figure \ref{Figure_Labelling_the_Rectangles_in_the_partititons_of_the_unit_square}). On the other hand, Lebesgue's space-filling curve is not connected \cite{bader2012space}.

\begin{theorem}
The space-filling curve $F$ constructed in Theorem \ref{t_fibonacci_sfc} is the only connected space-filling curve that is generated by decorated tiles of tiles of $\mu_2$ depicted in Figure \ref{Figure_Fibonacci_times_Fibonacci_Substitution}. That is, it is the only connected space-filling curve that can be generated by Cartesian product of two Fibonacci substitutions.
\end{theorem}

\begin{proof}
The proof follows by the fact that there is only one way of visiting tiles of 1-supertiles of $\omega$ (up to rotation) that inherits connectedness property. 
\end{proof}

We call the space-filling curve $F$ as the \emph{Fibonacci\ space-filling\ curve}.

\paragraph{Approximating Polygons of the Fibonacci Space-Filling Curve} We illustrate the iterative process of generating the Fibonacci space-filling curve $F$ through \emph{approximating\ polygons}, which is a notion introduced by Wunderlich \cite{wunderlich1973peano}.

\begin{definition}[Approximating Polygons]\label{d_nth-approximant-of-F}
With the same notations in the proof of Theorem \ref{t_fibonacci_sfc}, for each $k\in\ZZ^+$ and $i\in\{1,\dots,\cF^2_{k+2}\}$, denote the centre of $J^k_i$ by $x(J^k_i)$. The simple curve formed by joining points $x(J^k_1), x(J^k_2), \dots, x(J^k_{\cF^2_{k+2}})$ with straight lines respectively is called the \emph{k-th\ approximating\ polygon\ of\ $F$}.
\end{definition}
First four approximating polygons of $F$ are depicted in Figure \ref{Figure_Approximating_Polygons} and Figure \ref{figure_4th_approximating_polygon}.

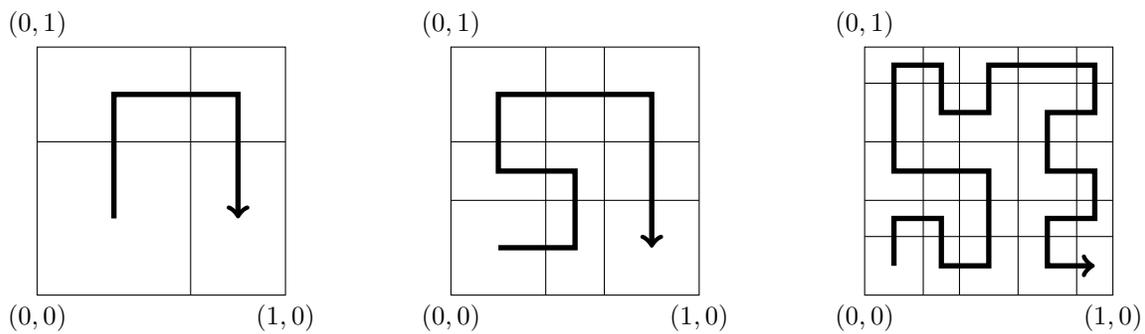
\begin{figure}[H]
\centering
\begin{tikzpicture}[scale=0.55]

\foreach \index in {0,10,20}{
\draw (0+\index,0) rectangle (6+\index,6);
\draw (\index+6/1.618,0) -- (\index+6/1.618,6);
\draw (\index,6/1.618) -- (\index+6,6/1.618);
\node[anchor=north] at (0+\index,0) {$(0,0)$};
\node[anchor=north] at (6+\index,0) {$(1,0)$};
\node[anchor=south] at (0+\index,6) {$(0,1)$};

}
\foreach \index in {10,20}{
\draw (\index+6/2.618,0) -- (\index+6/2.618,6);
\draw (\index,6/2.618) -- (\index+6,6/2.618);
}
\draw (20,6/4.236) -- (20+6,6/4.236);
\draw (20,6/4.236+6/1.618) -- (20+6,6/4.236+6/1.618);
\draw (20+6/4.236,0) -- (20+6/4.236,6);
\draw (20+6/4.236+6/1.618,0) -- (20+6/4.236+6/1.618,6);

\draw[line width=2pt, color=black, ->] (3/1.618,3/1.618) -- (3/1.618,3/1.618+3) -- (3/1.618+3,3/1.618+3) -- (3/1.618+3,3/1.618);
\draw[line width=2pt, color=black, ->] (10+3/2.618,3/2.618) -- (10+3/2.618+3/1.618,3/2.618) -- (10+3/2.618+3/1.618,3/2.618+3/1.618) -- (10+3/2.618,3/2.618+3/1.618) -- (10+3/2.618,3/2.618+3/1.618+3/1.618) -- (10+3/2.618+3/1.618+3/1.618,3/2.618+3/1.618+3/1.618) -- (10+3/2.618+3/1.618+3/1.618,3/2.618);
\draw[line width=2pt, color=black, ->] (20+3/4.236,3/4.236) -- (20+3/4.236,3/4.236+3/2.618) -- (20+3/4.236+3/2.618,3/4.236+3/2.618) -- (20+3/4.236+3/2.618,3/4.236) -- (20+3/4.236+3/2.618+3/2.618,3/4.236) -- (20+3/4.236+3/2.618+3/2.618,3/4.236+3/2.618+3/2.618) -- (20+3/4.236,3/4.236+3/2.618+3/2.618) -- (20+3/4.236,3/4.236+3/2.618+3/2.618+6/4.236+3/2.618) -- (20+3/4.236+3/2.618,3/4.236+3/2.618+3/2.618+6/4.236+3/2.618) -- (20+3/4.236+3/2.618,3/4.236+3/2.618+3/2.618+6/4.236) -- (20+3/4.236+3/2.618+3/2.618,3/4.236+3/2.618+3/2.618+6/4.236) -- (20+3/4.236+3/2.618+3/2.618,3/4.236+3/2.618+3/2.618+6/4.236+3/2.618) -- (20+3/4.236+3/2.618+3/2.618+6/4.236+3/2.618,3/4.236+3/2.618+3/2.618+6/4.236+3/2.618) -- (20+3/4.236+3/2.618+3/2.618+6/4.236+3/2.618,3/4.236+3/2.618+3/2.618+6/4.236) -- (20+3/4.236+3/2.618+3/2.618+6/4.236,3/4.236+3/2.618+3/2.618+6/4.236) -- (20+3/4.236+3/2.618+3/2.618+6/4.236,3/4.236+3/2.618+3/2.618) -- (20+3/4.236+3/2.618+3/2.618+6/4.236+3/2.618,3/4.236+3/2.618+3/2.618) -- (20+3/4.236+3/2.618+3/2.618+6/4.236+3/2.618,3/4.236+3/2.618) -- (20+3/4.236+3/2.618+3/2.618+6/4.236,3/4.236+3/2.618) -- (20+3/4.236+3/2.618+3/2.618+6/4.236,3/4.236) -- (20+3/4.236+3/2.618+3/2.618+6/4.236+3/2.618,3/4.236); 

\end{tikzpicture}
\caption{First three approximating polygons of $F$.}
\label{Figure_Approximating_Polygons}
\end{figure}

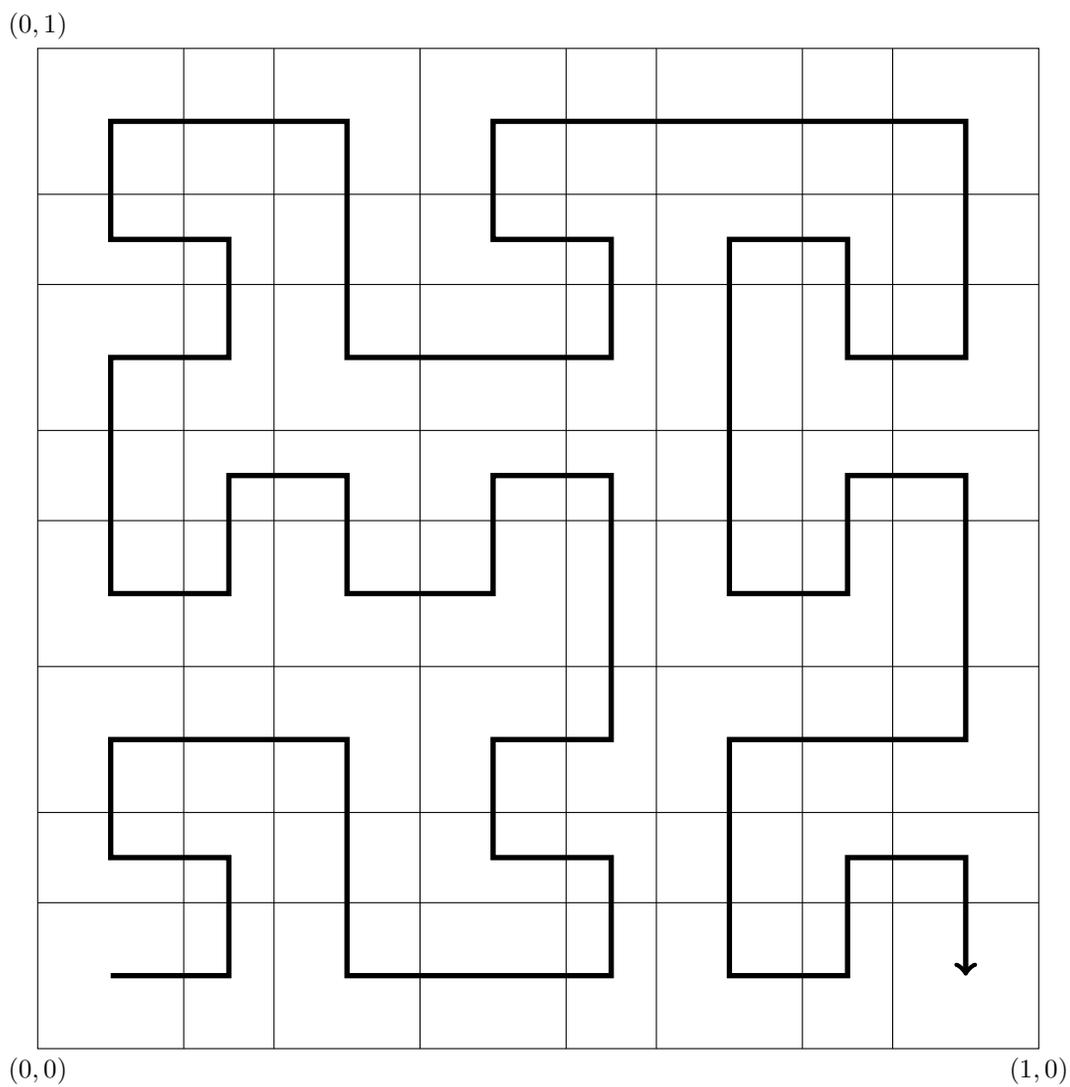
\begin{figure}[H]
\centering
\begin{tikzpicture}[scale=1.2]
\centering
\draw (10,0) rectangle (10+5*1.61803398875+3,5*1.61803398875+3);

\node[anchor=north] at (10,0) {$(0,0)$};
\node[anchor=north] at (10+5*1.61803398875+3,0) {$(1,0)$};
\node[anchor=south] at (10,5*1.61803398875+3) {$(0,1)$};

\foreach \index in {1.61803398875,2.61803398875,2*1.61803398875+1,3*1.61803398875+1,3*1.61803398875+2,4*1.61803398875+2,4*1.61803398875+3}{
\draw (10+\index,0) -- (10+\index,5*1.61803398875+3);
\draw (10+0,\index) -- (10+5*1.61803398875+3,\index);
}

\draw[line width=2pt, color=black, ->] (10+1.61803398875/2,1.61803398875/2) -- (10+1.61803398875+0.5,1.61803398875/2) -- (10+1.61803398875+0.5,1.61803398875+0.5) -- (10+1.61803398875/2,1.61803398875+0.5) -- (10+1.61803398875/2,1.61803398875/2+1.61803398875+1) -- (10+1.61803398875/2+1.61803398875+1,1.61803398875/2+1.61803398875+1) -- (10+1.61803398875/2+1.61803398875+1,1.61803398875/2) -- (10+3*1.61803398875+1.5,1.61803398875/2) -- (10+3*1.61803398875+1.5,1.61803398875+0.5) -- (10+2.5*1.61803398875+1,1.61803398875+0.5)-- (10+2.5*1.61803398875+1,1.5*1.61803398875+1) -- (10+3*1.61803398875+1.5,1.5*1.61803398875+1) -- (10+3*1.61803398875+1.5,3*1.61803398875+1.5) -- (10+2.5*1.61803398875+1,3*1.61803398875+1.5) -- (10+2.5*1.61803398875+1,2.5*1.61803398875+1) -- (10+1.5*1.61803398875+1,2.5*1.61803398875+1) -- (10+1.5*1.61803398875+1,3*1.61803398875+1.5) -- (10+1.61803398875+0.5,3*1.61803398875+1.5) -- (10+1.61803398875+0.5,2.5*1.61803398875+1) -- (10+0.5*1.61803398875,2.5*1.61803398875+1) -- (10+0.5*1.61803398875,3.5*1.61803398875+2) -- (10+1.61803398875+0.5,3.5*1.61803398875+2) -- (10+1.61803398875+0.5,4*1.61803398875+2.5) -- (10+0.5*1.61803398875,4*1.61803398875+2.5) -- (10+0.5*1.61803398875,4.5*1.61803398875+3) -- (10+0.5*1.61803398875,4.5*1.61803398875+3) -- (10+1.5*1.61803398875+1,4.5*1.61803398875+3) -- (10+1.5*1.61803398875+1,3.5*1.61803398875+2) -- (10+3*1.61803398875+1.5,3.5*1.61803398875+2) -- (10+3*1.61803398875+1.5,4*1.61803398875+2.5) -- (10+2.5*1.61803398875+1,4*1.61803398875+2.5) -- (10+2.5*1.61803398875+1,4.5*1.61803398875+3) -- (10+4.5*1.61803398875+3,4.5*1.61803398875+3) -- (10+4.5*1.61803398875+3,3.5*1.61803398875+2) -- (10+4*1.61803398875+2.5,3.5*1.61803398875+2) -- (10+4*1.61803398875+2.5,4*1.61803398875+2.5) -- (10+3.5*1.61803398875+2,4*1.61803398875+2.5) -- (10+3.5*1.61803398875+2,2.5*1.61803398875+1) -- (10+4*1.61803398875+2.5,2.5*1.61803398875+1) -- (10+4*1.61803398875+2.5,3*1.61803398875+1.5) -- (10+4.5*1.61803398875+3,3*1.61803398875+1.5) -- (10+4.5*1.61803398875+3,1.5*1.61803398875+1) -- (10+3.5*1.61803398875+2,1.5*1.61803398875+1) -- (10+3.5*1.61803398875+2,0.5*1.61803398875) -- (10+4*1.61803398875+2.5,0.5*1.61803398875) -- (10+4*1.61803398875+2.5,1.61803398875+0.5) -- (10+4.5*1.61803398875+3,1.61803398875+0.5) -- (10+4.5*1.61803398875+3,0.5*1.61803398875);
\end{tikzpicture}
\caption{Fourth approximating polygon of $F$. Figure is scaled up for demonstration purposes.} 
\label{figure_4th_approximating_polygon}
\end{figure}

\bibliographystyle{abbrv}
\bibliography{refs.bib}

\begin{thebibliography}{10}

\bibitem{bader2012space}
M.~Bader.
\newblock {\em Space-filling curves: an introduction with applications in scientific computing}, volume~9.
\newblock Springer Science \& Business Media, 2012.

\bibitem{b_cantor}
G.~Cantor.
\newblock Ein beitrag zur mannigfaltigkeitslehre.
\newblock {\em Journal f{\"u}r die reine und angewandte Mathematik (Crelles Journal)}, 84:242--258, 1878.

\bibitem{b_rao_zang_2}
X.-R. Dai, H.~Rao, and S.-Q. Zhang.
\newblock Space-filling curves of self-similar sets (ii): edge-to-trail substitution rule.
\newblock {\em Nonlinearity}, 32(5):1772, 2019.

\bibitem{b_dekking}
F.~M. Dekking.
\newblock Recurrent sets.
\newblock {\em Advances in mathematics}, 44(1):78--104, 1982.

\bibitem{frettloh2017substitution}
D.~Frettl{\"o}h, A.~L. Say-awen, and M.~L. A.~N. De~Las~Pe{\~n}as.
\newblock Substitution tilings with dense tile orientations and n-fold rotational symmetry.
\newblock {\em Indagationes Mathematicae}, 28(1):120--131, 2017.

\bibitem{b_hahn_2}
H.~Hahn.
\newblock Mengentheoretische charakterisierung der stetigen kurven.
\newblock {\em Sitzungsber Akad. Wiss. Wien}, 123:2433--2487, 1914.

\bibitem{b_hahn_1}
H.~Hahn.
\newblock {\"U}ber die allgemeinste ebene punktmenge, die stetiges bild einer strecke ist.
\newblock {\em Jahresbericht der Deutschen Mathematiker-Vereinigung}, 23:318--322, 1914.

\bibitem{b_hahn_3}
H.~Hahn.
\newblock {\"U}ber stetige streckenbilder.
\newblock {\em Atti del Congresso, Internazionale dei Mathematici, Bologna}, 6:217--220, 1928.

\bibitem{b_hausdorff_felix}
F.~Hausdorff.
\newblock {\em Mengenlehre 3rd edition}.
\newblock Berlin-Leipzig: de Gryter, 1927.

\bibitem{b_lebesgue}
H.~Lebesgue.
\newblock Le\c{c}ons sur l'int\'egration et la recherche des fonctions primitives.
\newblock {\em Gauthier-Villars, Paris}, pages 44--45, 1904.

\bibitem{b_lindenmayer}
A.~Lindenmayer.
\newblock Mathematical models for cellular interactions in development i. filaments with one-sided inputs.
\newblock {\em Journal of theoretical biology}, 18(3):280--299, 1968.

\bibitem{b_Mazurkiewicz_1}
S.~Mazurkiewicz.
\newblock O arytmetyzacji kontinuow.
\newblock {\em Sc. de Varsovie}, 6:305--311, 1913.

\bibitem{b_Mazurkiewicz_2}
S.~Mazurkiewicz.
\newblock O arytmetyzacji kontinuow ii.
\newblock {\em Sc. de Varsovie}, 6:941--945, 1913.

\bibitem{b_Mazurkiewicz_3}
S.~Mazurkiewicz.
\newblock Sur les lignes de jordan.
\newblock {\em Fundamenta Mathematicae}, 1:166--209, 1913.

\bibitem{b_moore}
E.~H. Moore.
\newblock On certain crinkly curves.
\newblock {\em Transactions of the American Mathematical Society}, 1:72--90, 1900.

\bibitem{b_netto}
E.~Netto.
\newblock Beitrag zur mannigfaltigkeitslehre.
\newblock {\em Journal f{\"u}r die reine und angewandte Mathematik (Crelles Journal)}, 86:263--268, 1879.

\bibitem{b_osgood}
W.~F. Osgood.
\newblock A jordan curve of positive area.
\newblock {\em Transactions of the American Mathematical Society}, 4:107--112, 1903.

\bibitem{b_ozkaraca}
M.~{\.I}. {\"O}zkaraca.
\newblock {\em An application of space filling curves to substitution tilings}.
\newblock PhD thesis, University of Glasgow, 2021.

\bibitem{ozkaraca_Lebesgue}
M.~{\.I}. {\"O}zkaraca.
\newblock Planar substitutions to lebesgue type space-filling curves and relatively dense fractal-like sets in the plane.
\newblock {\em Journal of Mathematical Analysis and Applications}, 530(2):127654, 2024.

\bibitem{b_peano}
G.~Peano.
\newblock Sur une courbe, qui remplit toute une aire plane.
\newblock {\em Math. Ann.}, 36:157--160, 1890.

\bibitem{b_polya}
G.~Polya.
\newblock Uber eine peanosche kurve.
\newblock {\em Bull. Acad. Sci. Cracovie (Sci. math. et nat. S{\'e}rie A)}, pages 305--313, 1913.

\bibitem{b_rao_zang_1}
H.~Rao and S.-Q. Zhang.
\newblock Space-filling curves of self-similar sets (i): iterated function systems with order structures.
\newblock {\em Nonlinearity}, 29(7):2112, 2016.

\bibitem{b_rao_zang_3}
H.~Rao and S.-Q. Zhang.
\newblock Space-filling curves of self-similar sets (iii): Skeletons.
\newblock {\em Fractals}, 28(02):2050028, 2020.

\bibitem{sadun2008topology}
L.~Sadun.
\newblock {\em Topology of tiling spaces}, volume~46.
\newblock American Mathematical Soc., 2008.

\bibitem{Journal_Sagan_1}
H.~Sagan.
\newblock Approximating polygons for lebesgue's and schoenberg's space filling curves.
\newblock {\em The American Mathematical Monthly}, 93(5):361--368, 1986.

\bibitem{Journal_Sagan_2}
H.~Sagan.
\newblock A geometrization of lebesgue’s space-filling curve.
\newblock {\em The Mathematical Intelligencer}, 15(4):37--43, 1993.

\bibitem{sagan}
H.~Sagan.
\newblock {\em Space-filling curves}.
\newblock Springer-Verlag, Berlin Heidelberg New York, 1994.

\bibitem{b_schoenberg}
I.~Schoenberg.
\newblock On the peano curve of lebesgue.
\newblock {\em Bull. Amer. Math. Soc}, 44(8):519, 1938.

\bibitem{b_sierpinski}
W.~Sierpinski.
\newblock Sur une nouvelle courbe continue qui remplit toute une aire plane.
\newblock {\em Bull. Acad. Sci. Cracovie (Sci. math. et nat. Serie A)}, pages 462--478, 1912.

\bibitem{b_willard_top_book}
S.~Willard.
\newblock {\em General Topology (Dover, New York, 2004)}.
\newblock Originally Published by Addison-Wesley Reading, MA, 1970.

\bibitem{wunderlich1973peano}
W.~Wunderlich.
\newblock {\"U}ber peano-kurven.
\newblock {\em Elemente der Mathematik}, 28:1--10, 1973.

\end{thebibliography}
\Addresses
\end{document}